\documentclass[a4 paper,10pt]{article}
\usepackage{amsmath}
\usepackage{latexsym}
\usepackage{amsfonts}
\usepackage{amssymb}
\usepackage{amsthm}
\usepackage{amsbsy}
\usepackage{graphicx}
\usepackage{bbm}
\usepackage{color}
\usepackage{epstopdf}

\newcommand{\cX}{\mathcal{X}}
\newcommand{\nada}[1]{}
\makeatletter
\def\@maketitle{\newpage
    \null
    \vskip .8truein
    \begin{center}%
     {\bf \@title \par}%
     \vskip 1.5em
     {\small
      \lineskip .5em
      \begin{tabular}[t]{c}\@author
      \end{tabular}\par}%
    \end{center}%
    \par
    \vskip .4truein}
\@addtoreset{equation}{section} 
\@addtoreset{theorem}{section} 
\@addtoreset{lemma}{section} \@addtoreset{proposition}{section}
\@addtoreset{definition}{section} \@addtoreset{corollary}{section}
\@addtoreset{remark}{section}

\let\G=\Gamma

\def\R{\mathbb{R}}
\def\Uno{\mathbbm{1}}

\def\G{\mathbb{G}}
\def\N{\mathbb{N}}
\def\dH{\mathbb{H}}
\def\Z{\mathbb{Z}}

\def\cH{\mathcal{H}}
\def\cX{\mathcal{X}}

\def\deps{\delta_{\frac{1}{\varepsilon}}}

\newcommand{\cg}{\mathrm{g}}

\makeatletter
\def\@maketitle{\newpage
    \null
    \vskip .8truein
    \begin{center}%
     {\bf \@title \par}%
     \vskip 1.5em
     {\small
      \lineskip .5em
      \begin{tabular}[t]{c}\@author
      \end{tabular}\par}%
    \end{center}%
    \par
    \vskip .4truein}

\newtheorem{teo}{Theorem}[section]
\newtheorem{theorem}{Theorem}[section]
\newtheorem{lemma}{Lemma}[section]
\newtheorem{proposition}{Proposition}[section]
\newtheorem{definition}{Definition}[section]
\newtheorem{defi}{Definition}[section]

\newtheorem{rem}{Remark}[section]

\newtheorem{ex}{Example}[section]

\newcommand{\norma}[1]{\left \|#1\right \|}

\def\proof{\list{}{\setlength{\leftmargin}{0pt}
                      \parskip=0pt\parsep=0pt\listparindent=2em
                      \itemindent=0pt}\item[]\futurelet\testchar\@maybe}

\def\@maybe{\ifx[\testchar \let\next\@Opt
          \else \let\next\@NoOpt \fi \next}
\def\@Opt[#1]{{\it Proof of #1.\ }}\def\@NoOpt{{\it Proof.\ }}

\begin{document}

\title{\Large \bf $\Gamma$- convergence and homogenisation for a class of degenerate functionals.}

\author{Nicolas Dirr\thanks{Cardiff School of Mathematics, Cardiff University, Cardiff, UK, e-mail: dirrnp@cardiff.ac.uk.} \and
Federica Dragoni
\thanks{Cardiff School of Mathematics, Cardiff University, Cardiff, UK, e-mail: DragoniF@cardiff.ac.uk.} \and
Paola Mannucci 
\thanks{Dipartimento di Matematica ``Tullio Levi-Civita'' ,  Universit\`a di Padova, Padova, Italy, e-mail: mannucci@math.unipd.it} \and
Claudio Marchi
\thanks{Dipartimento di Ingegneria dell'Informazione,  Universit\`a di Padova, Padova, Italy, e-mail: claudio.marchi@unipd.it}
}

\maketitle
\begin{abstract}
\noindent
This paper is on $\Gamma$-convergence for degenerate integral functionals related to homogenisation problems  in the Heisenberg group. Here both the rescaling and the notion of invariance or periodicity are chosen in a way motivated by the geometry of the Heisenberg group.  
Without using special geometric features, these functionals would be neither coercive nor periodic, so classic results do not apply.  All the results apply to the more general case of Carnot groups.
\end{abstract}

\noindent {\bf Keywords}:   Gamma-convergence, homogenisation,  Carnot groups, Heisenberg group, 
H\"ormander condition,
degenerate functionals.

\section{Introduction}

$\Gamma$-convergence is a notion of convergence of a family of functionals $F_\varepsilon$ to a  functional $F_{\infty}$  which goes back to E. de Giorgi (see \cite{DEGiorgi, DEGiorgi1, DEGiorgi2}) and guarantees the convergence of minimisers of the functionals $F_\varepsilon$ to minimisers of the limit (or effective) functional $F_{\infty}$; for a precise definition  and properties see Section \ref{SectionGamma}.The convergence of minimisers  implies, under suitable conditions, convergence of solutions of the Euler-Lagrange equations, and is therefore a useful tool for homogenisation problems, in particular in the random and nonlinear case, see e.g. \cite{DalMasoModica1}. If the functional is integral, i.e. of the form

$$
u\mapsto
F(u)=\int_{A}f(x,\nabla  u(x)) dx,
$$ where $A$ is a (Borel) domain, $\nabla u$ is the distributional gradient of the real-valued function $u$ assumed to be in a suitable $L^p$-space,  and $f:\R^N\times\R^N\to \R$ with some regularity and growth assumptions, then the corresponding Euler-Lagrange equation is a nonlinear divergence form equation.

 The above functional can be generalised to degenerate functionals in the setting of 
 Carnot groups. In this paper we focus specifically  on  the $n$-dimensional  Heisenberg group $\dH^n$, which is a step 2 Carnot group defined on $\R^{2n+1}$ (see Section \ref{Preliminaries} for definitions and properties).
Thus the family of functionals considered here is of the form

$$
u\mapsto
F(u)=\int_{A}f(x,\nabla_{\cX}  u(x)) dx
$$
where $A$ is a domain on $\R^N=\R^{2n+1}$ while $\nabla_{\cX}u$ is the horizontal gradient in the Heisenberg group
which belongs to a suitable $m$-dimensional subspace of the ``space of derivatives'' (tangent space),
see Definition \ref{horGradient}. Since $m=2n$, we have $m<N$:
as a consequence, such functionals are typically not coercive in the classical sense, so classical results do not apply.\\
Working in the setting of the Heisenberg group, the scaling  needs to adapt to the underlying geometrical structure, therefore we will consider the following anisotropic scaling 
 $\delta_{1/\varepsilon}(\underline{x}_1,\underline{x}_2,x_3)=(\varepsilon^{-1}\underline{x}_1,\varepsilon^{-1}\underline{x}_2,\varepsilon^{-2}x_3),$ with $(\underline{x}_1,\underline{x}_2,x_3)\in \R^N\equiv \R^n\times \R^n
\times \R$. 
 Hence the scaling is anisotropic w.r.t. the last component. The anisotropy can be understood heuristically in another way: at each point, some directions are ``forbidden'', i.e. paths of the associated control problem can move only on a $2n$-dimensional subspace of a $(2n+1)$-dimensional space. By varying their direction often (i.e. by the use of non-trivial commutators from the H\"ormander condition) they are able to reach any given point but the cost for ``zig-zagging'' to get in the forbidden direction is higher, so typically they move slower in these directions, which makes  a faster rescaling necessary.

The limit functional $F_{\infty}$ is of the same form,  i.e.
$$
 F_{\infty}(u)=\int_A f_0\left(\nabla_{\cX} u(x)\right)dx,
$$
where the integral function $f_0$ does not depend  on $x$ anymore
(however  the horizontal gradient still depends on $x$ trough the vector fields).
The corresponding Euler-Lagrange equation will not be elliptic  but only subelliptic,
we refer to  \cite{BLU, DomokosManfredi} for an overview on subelliptic equations.

The study of homogenisation in subelliptic settings started with the periodic case (see e.g. \cite{BW, BMT, Franchi1, Franchi2, Nicoletta, MS, Stroffolini}).
The first result for the stochastic case in this degenerate setting  is  \cite{DDMM}, where the authors studied the case Hamilton-Jacobi (first order) case for Hamiltonian depending on the horizontal gradient in the case of Carnot groups.\\

As $\Gamma$-convergence has nice compactness results, the main difficulty is in general the identification of the $\Gamma$-limit as again an integral functional. Here it is  used that the integrand can be retrieved by considering minimisation problems over small cubes with {\em affine } boundary conditions, see \cite{DalMasoModica2}. A generalisation to the setting of the Heisenberg group requires a suitable adaption of the notion of ``affine'', namely $H$-affine functions, see  Section \ref{Main result and notations.} for the definition and references. 
Recently some results for $\Gamma$-convergence of degenerate functionals in very general geometries have been proved in \cite{Maione}.
Here we use that the minimal normalised energy on anisotropicaly  (Heisenberg dilations)
scaled cubes  is  subadditive  by constructing admissible functionals  on large cubes and patching together translated minimisers  on translated cubes. Here we need  to use the specific properties of translations in the Heisenberg group. Note that 
 cubes rescaled by an integer (i.e. $\delta_{k}(Q)$)  cannot be written as union of translations of the original cube Q, not even up to a set of measure zero. This a crucial difference with the Euclidean  case  but we overcome the issue by controlling the error term. 
 
 A closely related approach can be found in \cite{Hafsa}, where the $\Gamma$-convergence in Cheeger-Sobolev spaces is considered. Our functional depending only on $|\nabla_{\cX}u|$ instead of
$\nabla_{\cX} u$ would be a functional on a Cheeger-Sobolev space, but the natural tiling generalising periodicity in our case does not satisfy the assumptions of \cite{Hafsa}.

All the results are written in the Heisenberg group for sake of simplicity but the proofs apply to general Carnot groups.

These results can be applied to functionals related to subelliptic $p$-Laplace equations and generalised to   stochastic functionals  with short correlations (as done in Dal Maso-Modica \cite{DMMampa}).

This paper is organised as follows.

 In Section 2 we  give an overview on the Heisenberg group and its geometry, in particular the scaling, the horizontal gradient and the notion of periodicity.

In Section 3 we define precisely our functionals and we recall the Sobolev
spaces adapted to the structure of the Heisenberg group, in particular their embedding into $L^p$-spaces, through the embedding in fractional  Sobolev spaces, which will be crucial for the later $\Gamma$-convergence results.

Section 4 is devoted to  the $\Gamma$-convergence  results. We first recall the definition and some basic properties of $\Gamma$-convergence. We use compactness properties of the $\Gamma$-convergence and we give conditions under which the $\Gamma$-limit  is again an integral functional, thus recovering the results by Dal Maso-Modica, \cite{DalMasoModica2}, for our degenerate functionals.

In Section 5 we prove the main result of the paper, that is the  homogenisation result for
Heisenberg-periodic functionals. In fact, we show a Akcoglu-Krengel
type result, \cite{AK}, for our anisotropic Heisenberg-periodic functionals, i.e. the convergence of normalised minimal energies over rescaled cubes. For this purpose, we exploit an underlying subadditive structure.

In Section 6 we mention some applications and further directions of research.  We  highlight how the results apply to more general functionals associated to Carnot group structures.
We then give some connections with homogenisation for subelliptic $p$-Laplacian.
 Finally we explain how our methods can be used to generalise the results to the 
stochastic case  with short correlations.

\section{Preliminaries: The Heisenberg group.}
\label{Preliminaries}
Carnot groups
 are non-commutative Lie groups: thus they are endowed both with a non-commutative algebraic structure and with a manifold structure. The lack of commutativity in the algebraic structure reflects on the manifold structure as restrictions on the admissible motions. This means that the allowed curves are constrained to have their velocities in a lower dimensional subspace of the tangent space of the manifold. Then the associated manifold structure is not Riemannian but sub-Riemannian. 
 In this paper we give details for the Heisenberg group only but the results can be easily generalised to Carnot groups (see Section 6). We
  refer the reader to \cite{BLU} for definitions and properties on Carnot groups and to
   \cite{montgomery} for an overview on
  sub-Riemannian manifolds.\\

To keep the paper easily readable  we omit the intrinsic definition of the Heisenberg group, introducing it directly as the following non-commutative group  structure on $\R^N$. 
\begin{defi}
\label{Heisenberg}
The $n$-dimensional Heisenberg group $\dH^n$, with $n\geq 1$, is a Carnot group of step 2
isomorphic to  $\R^N$, where  $N=2n+1$,
endowed with the  following non-commutative  group operation:
\begin{equation}
\label{Group_Law}
x* y:=\left(\underline{x}_1+\underline{y}_1, \underline{x}_2+\underline{y}_2,x_3+y_3+\frac{\underline{x}_1\cdot \underline{y}_2- \underline{x}_2\cdot \underline{y}_1}{2}\right)
\end{equation}
for all $x=(\underline{x}_1,\underline{x}_2,x_3),y=(\underline{y}_1,\underline{y}_2,y_3)\in \R^{N}\equiv \R^n\times \R^n\times \R$ and where by $\cdot$ we indicate the standard inner product in $\R^n$.
\end{defi}

In all Carnot groups it is possible to define  a natural scaling, induced by the Lie algebra stratification, namely dilations. The dilations replace the multiplication by scalars in the standard vector space structure of the Euclidean $\R^N$.
\begin{defi}
\label{Dilations_Heisenberg}
The dilations in the Heisenberg group are the family of group homeomorphisms defined as, for all $t>0$,  
$\delta_{t}:\R^N\to \R^N$  with
\begin{equation}
\label{Dialtions}
\delta_{t}(x)=(t\,\underline{x}_1,t\,\underline{x}_2, t^2\,x_3),
\quad \forall\;
x=(\underline{x}_1,\underline{x}_2,x_3)\in \R^N\equiv \R^n\times \R^n\times\R.
\end{equation}
\end{defi}
Thus  the dilations in $\dH^n$ coincide with the standard Euclidean scaling in the first $2n$ components while the last component scales as $t^2$.

The following properties of dilations are true in all Carnot groups and they can be easily checked in the Heiseberg group by using formulas \eqref{Group_Law} and  \eqref{Dialtions}.
\begin{lemma}
\label{DilationsProperties}
For all  $t, s >0$, the following properties hold true:
\begin{enumerate}
\item[(1)]  $\delta_{1} = \operatorname{id}$;
\item[(2)] $\delta^{-1}_{t} = \delta_{t^{-1}}$;
  \item[(3)]
  $\delta_{t}\circ \delta_{s} =\delta_{t \;s}$;
\item[(4)] for every $x, y\in \R^N$ one has $\delta_t(x) * \delta_t(y) = \delta_t(x* y)$.
\end{enumerate}
\end{lemma}
We now recall the notion of {\em homogenous dimension}. In a general Carnot group $\G$, the homogenous dimension is the natural number  
$\mathcal{Q}:=\sum_{i=1}^r i\, \textrm{dim}\, \cg_i$,
 where $r$ is the step of the  stratified  associated Lie algebra $\cg$
  (see e.g. \cite{BLU} for more details).
In $\dH^n$ one can easily show that
$$
\mathcal{Q}=2n+2.
$$
The homogeneous dimension  is correlated to the scaling of measures since it coincides with the Hausdorff dimension w.r.t. every homogeneous metric. 
 In the paper we always indicate simply by $|A|$ the $N$-dimensional Lebesgue measure of the Borel set $A$ of $\R^N$. Then for all $t>0$,
$|t\; A|=t^N|A|$ while
one can easily show that
\begin{equation}
\label{Volume}
A^t:=\delta_{t}(A)
\quad
\Rightarrow
\quad
|A^t|=t^\mathcal{Q} |A|.
\end{equation}

 Since the Heisenberg group (as all Carnot groups) is non-abelian, translations to the right or to the left determine two different families of homeomorphism on the group. As standard in this setting, we consider the {\em left-translations}, which are defined, for all $y\in \R^N$ as 
 $L_y:\R^N\to \R^N$ with
 $$
 L_y(x):=y* x,
 $$
where $*$ is the group operation defined in \eqref{Group_Law}.\\

Using the left-translations it is possible to define a sub-Riemannian structure on each Carnot group by introducing a suitable family of  left-invariant vector fields spanning to the first layer of the Lie algebra stratification. We omit the general definition on Carnot groups (see e.g. \cite{BLU}). In the specific case of the Heisenberg group, the vector fields can be found as 
\begin{equation}
\label{vectorFields_N}
X_i(x)=dL_x(e_i),
\end{equation}
 where $e_i$ are the unit vectors of the standard Euclidean basis on $\R^N$ for $i=1,\dots, 2n$. 
 One can also easily show that, for all $j=1,\dots,n$
 $$
 X_{2n+1}(x)=dL_x(e_{2n+1})
 =\big[
 X_j,X_{n+j}
 \big](x),
 $$
 where $\big[\cdot,\cdot\big]$ are the standard Lie brackets (called also commutators) defined for vector fields. 
 In the case $n=1$ the vector fields are 
\begin{equation}\label{vectorFields}
 X_1(x)=\left(\begin{array}{c}1 \\0 \\
 -\frac{x_2}{2}\end{array}\right)\quad
\textrm{and}
\quad 
X_2(x)=\left(\begin{array}{c}0 \\1 \\\frac{x_1}{2}\end{array}\right),
\quad \forall\, x=(x_1,x_2,x_3)\in \R^3.
\end{equation}
We recall that the previous vector fields are left-invariant by definition.
For later use we introduce the following simplified notation: given any function $u:\R^N\to \R$, the translation $L_z$ of the function  $u$ is simply $u\circ L_z$, i.e.
$$
L_y(u)(x):=u(y*x).
$$
Thus $X_i$ is a left-invariant vector field if   for all $u\in C^{\infty}(\R^N)$ and for all fixed $ y\in \R^N$
\begin{equation}\label{LeftInvariant_vectorFields}
X_i(  L_y(u))(x)=(X_iu)\,(y* x), \quad
\forall x\in \R^N,
\end{equation}
(while this is in general false considering instead the right-translations).\\
We recall that the vector fields $X_i$ for $i=1,\dots, 2n$ span a bracket generating distribution with step 2 (see e.g. \cite{montgomery} for some details).

The previous vector fields allow us to define  derivatives of any order, just considering how a vector field acts on smooth functions.  
 Given a function $u:\R^N\to \R$, we denote
 the \emph{horizontal gradient}  of $u$ 
  by 
\begin{equation}\label{horGradient}
\nabla_\cX u = (X_1 u,\dots, X_{2n}u)^T.
\end{equation}

In the case of $n=1$ the horizontal gradient can be explicitly  written as
$$
\nabla_{\cX} u=\begin{pmatrix}
u_{x_1}-\frac{x_2}{2}u_{x_3}\\
u_{x_2}+\frac{x_1}{2} u_{x_3}
\end{pmatrix}\in \R^2.
$$

We now recall that a differential operator $\mathcal{L}$ on a the Heisenberg group  is called homogeneous of degree $\kappa$ if for every $u\in C^\infty(\R^N)$ one has
\[
\mathcal{L}(\delta_t u) = t^\kappa \delta_t(\mathcal{L}u),
\]
where  the scaled function is defined as $\delta_t u(x) := u(\delta_t(x))$, for all $x\in \R^N$.

Then we  have the following result.

\begin{lemma}\label{L:xjhom}
For every $i = 1,...,2n$, each left-invariant vector field $X_i$, defined  in \eqref{vectorFields_N},
is  homogeneous of degree $\kappa = 1$, i.e., for any $u\in C^\infty(\R^N)$ one has
\[
X_i(\delta_t u) = t \delta_t(X_i u).
\]
\end{lemma}
The proof is a very simple computation in the Heisenberg group while for general Carnot groups the reader can find a proof e.g. in \cite{Fo}.\\

This in particular implies that
the horizontal 
gradient is homogeneous of degree one with respect to the dilations $\delta_\lambda$, i.e., for every $u\in C^\infty(\R^N)$ we have
\begin{equation}\label{laphom}
\nabla_{\cX} (\delta_t u)
= t
\delta_t\big( \nabla_{\cX} u\big).
\end{equation}

For later use,  
it is very useful to  introduce the $
N\times 2n$-matrix associated to the vector fields, that is
 \begin{equation}
\label{MatrixHeisenberg}
\sigma:=\big(X_1,\dots,X_{2n}\big),
\end{equation}
where 
 $X_{i}$ are the left-invariant vector fields defined in \eqref{vectorFields_N}
  and   the {\em extended matrix of vector fields}, which is the $
N\times N$-matrix
\begin{equation}
\label{MatrixHeisenberg_extended}
\sigma_{Ext}:=\big(X_1,\dots,X_{2n}, X_{2n+1}\big),
\end{equation}
where $X_{2n+1}(x)=dL_x(e_{2n+1})$, and $e_{2n+1}$ is the unit vector spanning the  $2n+1$-direction (and associated to the second layer of the stratification for the Lie algebra).
\begin{ex}
In the 1-dimensional Heisenberg group $\dH^1$, the matrix $\sigma$ is the $2\times 3$-matrix given by 
\begin{equation*}
\sigma(x_1,x_2,x_3)=
\begin{pmatrix}
1 &0\\
0&1\\
 -\frac{x_2}{2} &
\frac{x_1}{2}
\end{pmatrix},
\end{equation*}
while
$\sigma_{Ext}$ is the $3\times 3$-matrix given by 
\begin{equation*}
\sigma_{Ext}(x_1,x_2,x_3)=
\begin{pmatrix}
1 &0 &0\\
0&1&0\\
 -\frac{x_2}{2} &
\frac{x_1}{2} &1
\end{pmatrix}.
\end{equation*}
\end{ex}
A trivial computation shows the following property, which will be very useful later:
given the quadratic matrix defined in \eqref{MatrixHeisenberg_extended}, then 
\begin{equation}
\label{DeterminantMatrixExt}
\textrm{det}\big(\sigma_{Ext}(x)\big)=1,
\quad \forall\, x\in \R^N.
\end{equation}
The property above means that the left-translations are an  isometry for the associated $L^p$-spaces, i.e. informally setting for all fixed $z\in \R^N$ $y:=L_z(x)=z* x$, we have $dy=dx$.
\begin{rem}
Property \eqref{DeterminantMatrixExt} can be generalised to all Carnot groups in exponential coordinates or to more general Carnot-type groups (see e.g. \cite{BD1} for properties and definitions of Carnot-type groups). 
\end{rem}
\begin{rem}
\label{derivateSigma} 
Trivially $\nabla_{\cX}u=\sigma^T \nabla u$ 
where $\nabla u$ 
denotes   
the standard (Euclidean) gradient  
of $u$.

\end{rem}

\subsection{Periodicity in the Heisenberg group.}

Being the Heisenberg group a Lie group,
a very natural notion of periodicity can be introduced by left-translations,
see 
\cite{{BW},{BMT},{Franchi1},{Franchi2}}.
We refer also to the Phd thesis \cite{Jama} where periodicity in the Heisenberg group (but also in more general structures as Grushin spaces) is studied in  details with many properties and examples.
Given any $\Omega\subset \R^N$, we say that $\Omega$ is H-periodic with period $T>0$,
 whenever $L_{Tk}(\Omega)=(Tk)* \Omega = \Omega$ for all $k\in \Z^N$.\\
 For later use in the paper we  fix the period $T=2$. In fact, recalling that $L_{y}\circ \, L_{z}(x)=
 L_{y* z}(x)$,  the composition of two left-translations with period $T=1$  is not anymore a integer left-translation since $k* h \notin \Z^N$, because the third component becomes
 $
 k_3+h_3+
 \frac{k_1h_2-k_2h_1}{2}
 $, which is in general not anymore an integer.

 Instead the composition of two left-translations with period 2 is still a translation of the same type since,
  for all $k, h\in \Z^N$,
 $2k* 2h=2z$ 

 with $z=\left(\underline{k}_1+\underline{h}_1,\underline{k}_2+\underline{h}_2,
 k_3+h_3+\underline{k}_2\cdot
 \underline{h}_1
 -\underline{k}_1\cdot
 \underline{h}_2
 \right)\in \Z^N$.
 (Note that $z\neq k*h$ since the third component is different by a factor $\frac{1}{2}$ in the mixed term.)
 One could very simply  adapt everything to period $T=1$ by choosing a different representation of the Heisenberg group, where the group law is expressed by polynomials with integer coefficients; in that case the unit cell needs to be rescaled   to a unit cube (e.g. $[-\frac{1}{2},\frac{1}{2}]^N$), see e.g. \cite{Franchi1}.
 
 We introduce the following simplified notation for the left-translations with period 2, that is
\begin{equation}
\label{TranslazioneImportante}
\tau_k(x):=2k* x,
\quad \forall\,k\in \R^N, \; x\in \R^N.
\end{equation}
We recall that, for all $k,h\in \R^N$, the following properties hold true:
$$
 \tau_k\circ\tau_h
=\tau_{k* h}\quad
\textrm{and}
\quad
\tau_k^{-1}=\tau_{-k}.
$$

A definition of periodicity adapted to the Heisenberg group structure can be given for functions as follows.
\begin{defi}
\label{DefPeriodicFunction}
We say that the function
$f:\R^N\to \R$  is $H$-periodic  whenever
$$
f(\tau_k(x))=f(2k* x)=f(x),
\quad 
\forall\; x\in \R^N, k\in \Z^N.
$$
\end{defi}
To construct a large class of periodic functions we need to introduce a H-periodic tiling of $\R^N$.
Thus we consider the semiopen cube $Q=[-1,1)^N$. We call $Q$ {\em unit cell} and consider
$\tau_k(Q)
=2k*Q$.
Then one can easily show that the family
$
\big\{\tau_k(Q)\big\}_{k\in \Z^N}
$
fullfills
\begin{equation}
\label{propertytiling}
\bigcup_{k\in \Z^N}\tau_k(Q)=\R^N
\quad
\textrm{and}\quad
\tau_k(Q)\cap
\tau_{h}(Q)=\emptyset,
\; \forall\, k\neq h.
\end{equation}
see Figure \ref{Fig_tilining_Projected}  and   \cite[Lemma 2.4]{Franchi2}.

\begin{figure}[htbp]
\centerline{
\includegraphics[scale=0.55]{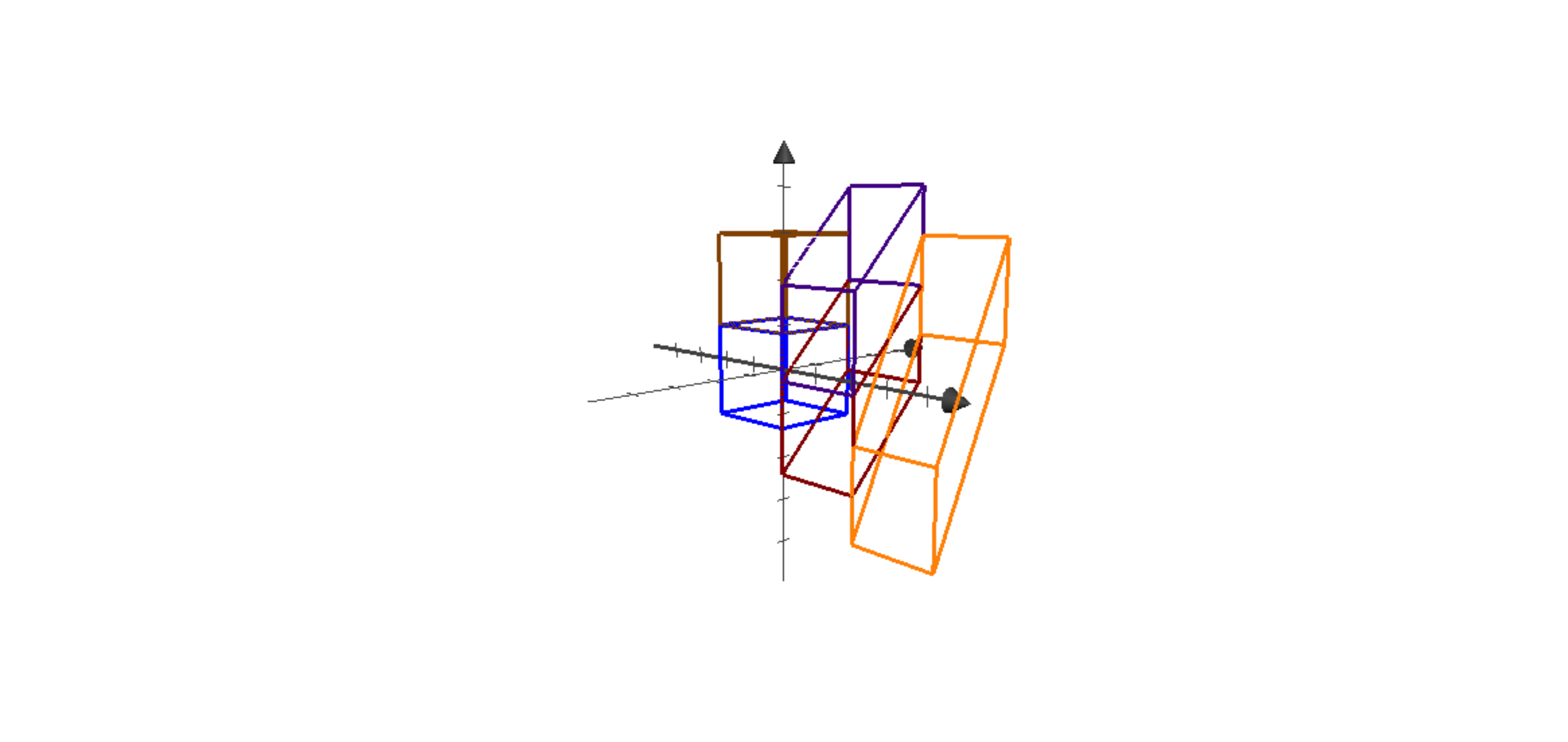}}
\caption{Tiling in $\dH^1$ constructed  by translating  $Q=[-1,1)^3$.}
\label{Fig_tilining_Projected}
\end{figure}

\begin{figure} [htbp]
\begin{center}
\includegraphics[scale=0.33]{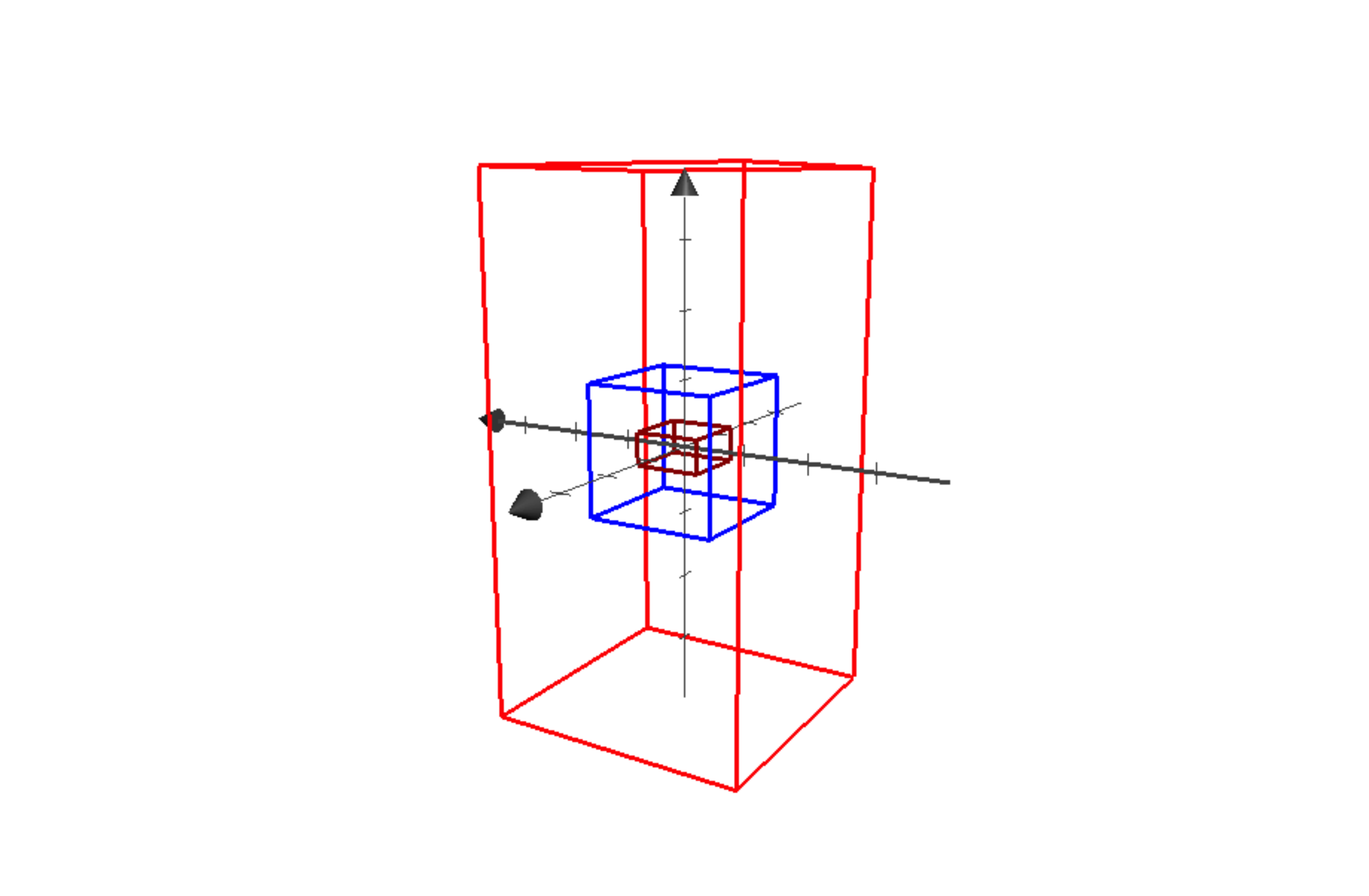}
\caption{Rescaling of the unit cell $Q=[-1,1)^3$ (which is the blue cube)  w.r.t. the dilations in the 1-dimensional Heisenberg group: in particular in red one can see $\delta_2(Q)$ while in bordeaux one can see $\delta_{\frac{1}{2}}(Q)$.}
  \label{Figura-Dilations}
  \end{center}  
\end{figure}

We next want to highlight a few facts about the scaling of tilings since this will be crucial later when we will study our homogenisation problem.
First recall that, in the Heisenberg group, if we scale the unit cell, then we do not get anymore hypercubes but hyper-rectangles since the scaling is anisotropic. In Figure \ref{Figura-Dilations} we show how the cube $Q$ scales for $t>1$ and for $t<1$.
Then if we want to build a tiling of $\R^N$  starting by a rescaled cell, we need to be very careful and adapt the translations to the Heisenberg  scaling.
\begin{lemma}
Given the unit cell  $Q=[-1,1)^N$ and a $t>0$,
the scaled unit cell as $$Q^t=\delta_t(Q)\quad \textrm{and}\quad Q^t_k:=\tau_{\delta_t(k)}\big(Q^t\big),$$
then the family
$
\big\{Q^t_k\big\}_{k\in \Z^N}
$
is  a tiling of $\R^N$ in the sense that
\begin{equation}
\label{propertytiling_rescaled}
\bigcup_{k\in \Z^N}Q^t_k
=\R^N
\quad
\textrm{and}\quad
Q^t_k\cap
Q^t_h=\emptyset,
\; \forall\, k\neq h.
\end{equation}
\end{lemma}
\begin{proof}
The result follows easily from the properties for $t=1$ and  from the fact that
$\delta_t(k)*\delta_t(Q)=\delta_t(k*Q).$
\end{proof}

\section{A class of degenerate functionals.}
\label{Main result and notations.}

Affine functions can be introduced in different ways in the Heisenberg group setting and they have been  studied in \cite{{D-G-N},{BD2}}.
For the purpose of the paper, we say that a function $u:\R^N\to \R$ is $H$-affine (in the Heisenberg group) if
$$
u(x)=q\cdot \pi_{m} (x)+a,
$$
 for  $m=2n$ and for some $q\in \R^{m}$ and $a\in \R$,  where $\pi_{m} :\R^{N}\to \R^{m}$ is the projection on the first $m$ components and $\cdot$ is the standard inner product on $\R^{m}$.   The following lemma is an immediate  property of $H$-affine functions in all Carnot-type groups and it will be key  for  our later results.
 \begin{lemma}
 \label{affineLemma} For all fixed $q\in \R^{m}$, we have
 $$
 \nabla_{\cX} u=q
 \quad
 \Leftrightarrow
 \quad 
 u(x)=q\cdot \pi_{m} (x)+a,
 $$
 for some $a\in \R$ and for all $x\in \R^N$.
 \end{lemma}
 \begin{proof}
 One implication (from the right to the left) 
 follows trivially from the fact that  $q\cdot \pi_{m} (x)+a$ does not depend on the last coordinate and the structure of the horizontal gradient.\\ 

 The other implication follows from the fact that $ \nabla_{\cX} u=q$ means $X_iu(x)=q_i=$constant for all $i=1,\dots,m$, then 
 $$
 u_{2n+1}=[X_1,X_{n+1}]u =X_1(X_{n+1}u)-X_{n+1}(X_1u)=X_1(q_{n+1})-X_{n+1}(q_1)=0,
 $$
 where we indicate by $u_i$ the partial derivative of $u$ w.r.t. the variable $x_i$, for $i=1,\dots,N$.
Using $u_{2n+1}(x)=0$, for all $x\in \R^N$, $X_iu(x)=q_i$ implies $u_{i}(x)=q_i$ for all $i=1,\dots, 2n$, which gives
 $u(x)=q\cdot \pi_{m} (x)+a$  for some $a\in \R$.
 \end{proof}
 We will later often use the following notation  for H-linear functions:
\begin{equation}
\label{H-linear}
 l_q(x)=q\cdot \pi_{m} (x).
\end{equation}

We next recall that the definition of Sobolev spaces in the setting of H\"ormander vector fields, which applies in particular to the Heisenberg group. We refer to \cite{rs, Xu-sem-94} 

 for more details on these spaces.

Let $k>1$ be an integer, $1\le \alpha \le + \infty $ and $A$ a domain on $\R^N$. We define the space 
$$
W_{\cX}^{k,\alpha}(A) = 
\left\{  
u \in L^\alpha(A) \;\big|
 \; \cX^J u \in L^\alpha(A) , \; \forall J \in \N^m, \; |J| \le k  \right\}
$$
where $\cX^Ju=X_1^{J_1}\dots  X_{m}^{J_{m}}u$ for $J=(J_1,\dots,J_{m})$.
Endowed with the norm
$$
\norma{u}_{ W_{\cX}^{k,\alpha} (A)}
= \left( \sum_{|J|\le  k }  \int_{A } 
|\cX^J u |^\alpha \, dx  \right)^{1/\alpha},   
$$
$W_{\cX}^{k,\alpha}(A)$ is a Banach space, and is an Hilbert space in the case $\alpha=2$.\\

Moreover, for any $1\le \alpha<+ \infty$, the embeddings
$$
W_{\cX}^{k, \alpha}(A) \hookrightarrow W^{k/r, \alpha}(A) \,,
$$
 hold true, where $r$ is the step of the  stratified  associated Lie algebra, thus in Heisenberg group $r=2$
 (see e.g. \cite{Xu90-var}). 
Later we will also need the following compact embedding. 
\begin{lemma}\label{Lp-emb}
 $W_{\cX}^{k, \alpha}(A)$
is compactly embedded into $L^{\alpha}(A)$. 
\end{lemma}
\begin{proof}
This follows from the previous embedding and the fact that the 
fractional Sobolev space  $W^{k/r, \alpha}(A)$ is compactly embedded into $L^p(A)$ (see e.g. \cite{DPV12}).  
\end{proof}

\begin{defi}
For each domain $A\subset \R^N$, we indicate by
$$
W_{\cX,0}^{k,\alpha}(A)
$$
the closure of $C^{\infty}_0(A)$ w.r.t. the Sobolev norm
$\norma{\cdot}_{ W_{\cX}^{k,\alpha} (A)}$.
\end{defi}
 This means that, whenever the  boundary $\partial A$ is regular enough, the trace of $u$ vanishes on the boundary of the set.\\
We will use this notation to express the Dirichlet boundary conditions:
 more precisely
$$
u-u_0\in W_{\cX,0}^{k,\alpha}(A)
$$
are all the functions 
$u\in W_{\cX}^{k,\alpha}(A)$ 
   which coincide  on $\partial A$
   (in the sense of Sobolev space) 
   with  some $u_0\in W_{\cX}^{k,\alpha}(A)$.   
 
 Next we recall the following Poincar\'e inequality, which is key for later results.
 
 \begin{lemma}
 \label{poincare}
 Given a bounded domain $A\subset \R^N$,
 then there exists a constant $C>0$ such that
 $$
 \int_A |u|^{\alpha} dx\leq C
  \int_A |\nabla_{\cX} u|^{\alpha} dx,
  \quad
    \forall\;  u\in W^{1,\alpha}_{\cX,0}(A).
 $$
 \end{lemma}
 \begin{proof}
 This follows from the results in \cite{Xu90-var, Lu}.
 \end{proof}

 Consider now a function  $f:\R^N\times \R^m\to \R$ with $N=2n+1$ and $m=2n$, we introduce the integral functional defined, for all  domain $A\subset \R^N$, as
\begin{equation}
\label{F}
F(u,A):=
\left\{\begin{aligned}
\int_A f\left(x,\nabla_{\cX} u\right)\; d\, x,&
\quad
u\in  W_{\cX}^{1,\alpha} (A),\\
+\infty,& \quad\textrm{else}.
\end{aligned}
\right.
\end{equation}
We introduce  the following properties for the integrand function\\
$f:\R^N\times \R^m\to \R$ (with $N=2n+1$ and $m=2n$)
\begin{equation}
\label{assumption1}
f(x,q) \; \textrm{is measurable in}\; x  \; \textrm{and convex in}\; q;
\end{equation}

\begin{equation}
\label{assumption2}
\begin{aligned}
&\exists\;
C_1,C_2>0\; \textrm{and}\; \alpha>1 \; \textrm{such that}\\
&C_1|q|^\alpha\leq
f(x,q)
\leq C_2\left(|q|^\alpha+1
\right),\quad q\in \R^m,\; x\in \R^N.
\end{aligned}
\end{equation}

Moreover for the later homogenisation problem we will assume $H$-periodicity for the functional in the sense of Definition \ref{DefPeriodicFunction}; more precisely
\begin{equation}
\label{assumption-Periodic}
f(\tau_{k}(x),q)=f(2k* x,q)=f(x,q),
\quad \forall\; x\in \R^N,\,q\in \R^m,\,k\in \Z^N.
\end{equation} 

\begin{ex}
\label{MainModel}
The main example is $f(x,q)=a(x)|q|^{\alpha}$, which trivially satisfies \eqref{assumption1} and \eqref{assumption2} whenever $a:\R^N\to \R$ is bounded (with a strictly positive lower bound) and measurable,
while assumption \eqref{assumption-Periodic} is equivalent to requiring that  $a(\cdot)$ is H-periodic.
\end{ex}

We want to study the minimisation problem for $F(u,A)$ with $H$-affine boundary condition, i.e.
\begin{equation}
\label{min}
m(F,u_0,A):=\min
\left\{F(u,A)\,|\, u-u_0\in W_{\cX,0}^{1,\alpha}(A)
\right\},
\end{equation}
with $u_0(x)=q\cdot \pi_{m}(x)+a$ for some  $q\in \R^m$ and $a\in \R$.\\

\begin{rem}
Note that  under assumptions \eqref{assumption1} and \eqref{assumption2},  the infimum of $F(u,A)$ on the set of  functions $u$ such that  
$u-u_0\in W_{\cX,0}^{1,\alpha}(A)$ is indeed a minimum by standard arguments, using the convexity, the embedding in Lemma \ref{Lp-emb} and the  Poincar\'e inequality (see Lemma~\ref{poincare}).
\end{rem}

\section{A $\Gamma$-convergence result for degenerate functionals.}
\label{SectionGamma}
To keep the paper self contained we next  recall briefly the definition of $\Gamma$-convergence.

\subsection{Very brief introduction to $\Gamma$-convergence.}
In homogenisation theory, we consider a family of solutions to equations with rapidly oscillating coefficients and investigate if they converge to a solution of a homogenised equation with slowly oscillating or constant coefficients. If these equations are the Euler-Lagrange equations of a suitable family of functionals with rapidly oscillating coefficients, and if both minimisers and solutions of the Euler-Lagrange equation are unique, then we can study convergence of the family of functionals instead. 

We need a notion of convergence of functionals which guarantees that minimisers of the approximating functionals converge to minimisers of the limit functional.

A suitable mathematical setup to make this rigorous is
the notion of $\Gamma$-convergence.
Let us  briefly recall the definition of $\Gamma$-convergence (see~\cite{BraDef, Bra, DM} for more details on this subject).

\begin{definition}\label{defgammac}
Let $U$ be a metric space and for $\varepsilon >0$ let $F_\varepsilon: U\to \overline{\R}$
be a family of functionals on $U$. 
We say that $F_\varepsilon$ 
$\Gamma$-converge to $F: U\to \overline{\R}$ 
if the following conditions are verified:
\begin{enumerate}
\item for all $u\in U$ and for all $u_\varepsilon\to u$, there holds 
$\displaystyle{\liminf_{\varepsilon\to 0}}F_\varepsilon(u_\varepsilon)\ge F(u)$ \\ \emph{($\Gamma$-liminf inequality);}
\item for all $u\in U$ there exist $u_\varepsilon\to u$, such that 
$\displaystyle{\lim_{\varepsilon\to 0}}F_\varepsilon(u_\varepsilon)= F(u)$ \\ \emph{($\Gamma$-limsup inequality).}
\end{enumerate}
\end{definition}

The convergence of minimisers to minimisers is formalised in the following way, 
which can be easily derived from Definition~\ref{defgammac}.

\begin{proposition}\label{propgammac}
If $F_\varepsilon$ $\Gamma$-converge to $F$ in $U$, also the corresponding minimal values (or infima) 
converge. Moreover, if $u_\varepsilon$ is a minimiser of $F_\varepsilon$ and $u_\varepsilon\to u\in U$, 
then $u$ is a minimiser of $F$. 
\end{proposition}

Hence, the asymptotic behaviour of minimisers of $F_\varepsilon$ (and therefore solutions of the Euler-Lagrange equations, see Section 6) can be partly understood by considering 
the $\Gamma$-limit of $F_\varepsilon$. 

Moreover, $\Gamma$-convergence has nice compactness properties, i.e. in general it is easy to show that a $\Gamma$-limit along subsequences exists. The problem is then to identify this limit (see Section 5) and to show properties of this limit, in particular that it is again an integral functional.

\subsection{$\Gamma$-convergence limit.}

We say that a family
$(A_\rho)_{\rho>0}$ of open subsets of~$\R^N$ with Lipschitz boundary  is a {\it substantial family} (around~$x$) as~$\rho\to0^+$ if, for every positive~$\rho$, there hold
\[
A_\rho \subset B(x,\rho) := \{y\in\R^N\,\big|\, |y-x|<\rho\}\quad 
\textrm{and}
\quad
|A_\rho|\geq c\;|B(x,\rho)|,
\]
where~$c$ is a constant independent of~$\rho$ (see the monograph~\cite[Ch.8]{rud} for other properties). 

The following result states that the integral function~$f$ can be obtained from the minima of the Dirichlet problem for 
$F$ with affine boundary data. To this purpose, for any domain
$A$ with Lipschitz boundary and for every $u_0$ $H$-affine data, we introduce the following regularised variational problem
\begin{equation}\label{min_reg}
m_{reg}(F,u_0,A):=\inf
\left\{F(u,A)\,|\, u\in C^\infty(\R^N),\, u=u_0\textrm{ on }\partial A
\right\}.
\end{equation}
Since the functional depends on $u$ only through its horizontal gradient, therefore the constant $a$ in the definition of $H$-affine function does not affect the results; we now consider directly the
$H$-linear functions defined in \eqref{H-linear} as boundary data.
We state now a useful property, key for the later results.

\begin{lemma}
\label{Lemma1}
Let $A$ be a $N$-dimensional domain with Lipschitz boundary.
For $q\in\R^{m}$ and
$l_q(y)=q\cdot \pi_{m}(y)$,  and for every smooth function $u$ such that $u=l_q$ on $\partial A$,
we have
\begin{equation}\label{cubi1}
\int_{A}\nabla_\cX u\, dy = q\, |A|.
\end{equation}
\end{lemma}
\begin{proof}
We write $y=(\underline{y}_1,\underline{y}_2,y_3)\in\R^N\equiv \R^n\times\R^n\times \R$.
First
we show that
\begin{equation}\label{div}
\int_{A}\nabla_\cX u\, dy =  \int_{\partial A}(q\cdot\pi_{m})\;\nu_0\,d\,\cH^{N-1}\in \R^m,
\end{equation}
where $\nu_0$ is the horizontal normal, i.e. $\nu_0=\sigma^T \nu$ with $\sigma$ matrix of the vector fields  defined in \eqref{MatrixHeisenberg} and $\nu$ the outward unit normal to $\partial A$, while $d\,\cH^{N-1}$ is the Hausdorff measure defined on $\partial A$. 
To prove the claim \eqref{div}
we use that the vector fields  in the Heisenberg group are divergence free
  and we combine a simply integration by parts  with Remark \ref{derivateSigma}, which gives
  $$\int_{A}X_i u\, dy=
\int_{\partial A}u \;
\nu^i_0\;
d\,\cH^{N-1}
= \int_{\partial A}(q\cdot \pi_{m})\;\nu^i_0\;d\,\cH^{N-1},
$$
where $\nu^i_0$ is the $i$-component of $\nu_0$  for $i=1,\dots, m=2n$.

Then we can use the divergence theorem again (together with the fact that the vector fields in Carnot groups are divergence free) to conclude:
$$
 \int_{\partial A}(q\cdot \pi_{m})\;\nu^i_0\;d\,\cH^{N-1}=
 \int_A  X_i \left(q\cdot \pi_{m}\right)\, d\, x= 
q_i  \int_A d\, x=
 q_i |A|,
$$
for all $i=1,\dots, m$.

\end{proof}
 We now use the previous lemma to show that, whenever the integrand function $f$ does not depend on $x$, then  $H$-affine functions are minimisers for 
 problem \eqref{min} with $H$-affine boundary condition.
 \begin{lemma}
 \label{Lemma2}
 Given a domain with Lipschitz boundary $A\subset \R^N$, consider the problem \eqref{min} with $u_0(x)=q\cdot \pi_m(x)+a $ for some $q\in \R^m$ and $a\in \R$,
and  $F$ defined in \eqref{F} with $f(x,q)=f(q)$ convex, then
 $$
 m_{reg}(F,u_0,A)=\int_Af(\nabla_{\cX}u_0)\, d\,x=f(q)|A|.
 $$
 \end{lemma}
\begin{proof}
The result follows by combining Lemmas \ref{affineLemma} and  \ref{Lemma1} with Jensen's inequality.
\end{proof}

Now following the arguments of Dal Maso-Modica \cite{DalMasoModica2} for the standard non degenerate case, we prove that the integrand function $f$ can be retrieved in term of $m_{reg}$.

\begin{theorem}\label{T1rendic}
Under assumptions~\eqref{assumption1}-\eqref{assumption2},
there exists a measurable subset~$\cal N$ of~$\R^N$ with $|{\cal N}|=0$ such that
\[
f(x,q)=\lim_{\rho\to0^+}\frac{m_{reg}(F,l_q,A_{\rho})}{|A_\rho|},
\]
for every~$q\in \R^{m}$ and with $l_q$ $H$-linear as in \eqref{H-linear}, $x\in \R^N\setminus {\cal N}$ and every substantial family~$(A_\rho)_{\rho>0}$ around $x$.
\end{theorem}
\begin{proof}
We will use the same arguments as in~\cite[Theorem I]{DalMasoModica2} for the non-degenerate case.
We just sketch the main steps.

\noindent \texttt{Step 1.} Let us at the moment  assume  that there exists some $R>0$ such that~$f$ does not depend on~$x$ for~$|q|>R$.
Then by Jensen's inequality and by Lemma~\ref{Lemma1}, we obtain
\begin{eqnarray*}\notag
\inf \left\{\int_{A_\rho}f(x,\nabla_\cX u)\, dy\,|\, u\in C^\infty(\R^N),\, u=l_q\textrm{ on }\partial A_\rho \right\}&=&\int_{A_\rho}f(x,q)\, dy \\
 &=&  |A_\rho|f(x,q),
\end{eqnarray*}
for every $q\in\R^{m}$, $\rho>0$, $x\in\R^N$ and $l_q$ defined by \eqref{H-linear}.
Exactly as in~\cite[Proposition 1.1]{DalMasoModica2}, we can deduce
\begin{equation*}
\left|f(x,q)-\frac{m_{reg}(F, l_q, A_{\rho})}{|A_\rho|}\right|\leq
\frac{1}{|A_\rho|}\int_{A_\rho}\sup_{q\in\R^{m}}|f(x,q)-f(y,q)|\, dy.
\end{equation*}
It remains to prove that there exists a measurable subset~$\cal N$ of~$\R^N$ with $|{\cal N}|=0$ such that
\[
\lim_{\rho\to0^+}\frac{1}{|A_\rho|}\int_{A_\rho}\phi(x,y)\, dy=0,
\]
for every~$x\in\R^N\setminus{\cal N}$ and every substantial family~$\{A_\rho\}$ around $x$,
where
\[
\phi(x,y):= \sup_{q\in\R^{m}}\omega(x,y,q),\quad \omega(x,y,q):=|f(x,q)-f(y,q)|,\quad\forall x,y\in\R^N,q\in\R^{m}.
\]
We observe that $\omega(x,y,q)=0$ for every~$q\in\R^{m}$ with~$|q|\geq R$ and that, arguing as in~\cite{DalMasoModica2} (recall that $f$ is convex w.r.t.~$q$), there exists a positive constant~$K$ such that: $|\omega(x,y,q_1)-\omega(x,y,q_2)|\leq K|q_1-q_2|$ for every $x,y\in\R^{N}$ and $q_1,q_2\in\R^{m}$.

Fix a dense subset~$D$ of~$\R^{m}$. The Lebesgue's Differentiation Theorem ensures that there exists a measurable subset~$\cal N$ of~$\R^N$ with $|{\cal N}|=0$ such that
\[
\lim_{\rho\to0^+}\frac{1}{|A_\rho|}\int_{A_\rho}\omega(x,y,p)\, dy=0,\qquad\forall x\in\R^N\setminus{\cal N},\ p\in D,
\]
for every substantial family~$\{A_\rho\}$  around~$x$. Moreover, as in~\cite{DalMasoModica2}, for every~$\varepsilon>0$, there exists a finite set $\{p_1,\dots,p_k\}\subset D$ such that
\[
\phi(x,y)\leq \sum_{i=1}^{k}\omega(x,y,p_i)+K\varepsilon, \qquad\forall x,y\in\R^N.
\]
Therefore, we infer
\[
\limsup_{\rho\to0^+} \frac{1}{|A_\rho|}\int_{A_\rho}\phi(x,y)\, dy\leq K\varepsilon,
\qquad \forall x\in\R^N\setminus{\cal N},
\]
for every substantial family~$\{A_\rho\}$  around~$x$.
By the arbitrariness of~$\varepsilon$, we accomplish the proof.

\noindent \texttt{Step 2.}
Let us now remove the additional assumption of Step~$1$. Taking into account the convexity and the coercivity of~$f$ w.r.t.~$q$, by the same arguments as in~\cite[Theorem I]{DalMasoModica2}, we obtain that there exists a measurable set~${\cal N}'\subset \R^N$, with~$|{\cal N}'|=0$, such that
\[
f(x,q)\geq \limsup_{\rho\to 0^+} \frac{m_{reg}(F,l_q,A_{\rho})}{|A_\rho|},\qquad \forall x\in\R^N\setminus {\cal N}',\, q\in \R^{m},
\]
for every family~$(A_\rho)_{\rho>0}$ as in the statement.
In order to obtain the reverse inequality, we first observe that the same arguments of~\cite[Lemma 1.2]{DalMasoModica2} ensure that there exists an increasing sequence~$\{f_i\}_{i\in\N}$ of functions such that $f=\sup_if_i$ and each~$f_i$ satisfies the assumptions of step~$1$. For each~$i\in\N$, we denote~$F_i$ and~${\cal N}_i$ respectively the corresponding functional and the negligible set given by step~$1$. We set ${\cal N}'':=\cup_{i=1}^{\infty}{\cal N}_i$.
Step~$1$ for~$f_i$ and the inequality~$f\geq f_i$ entail
\[
f_i(x,q)=\lim_{\rho\to0^+}\frac{m_{reg}(F_i,l_q,A_{\rho})}{|A_\rho|}\leq
\liminf_{\rho\to0^+}\frac{m_{reg}(F,l_q,A_{\rho})}{|A_\rho|},\quad \forall x\in\R^N\setminus{\cal N}'', q\in \R^{m},
\]
for every family~$(A_\rho)_{\rho>0}$ as in the statement. Passing to the limit as~$i\to+\infty$, one deduces
\[
f(x,q)\leq
\liminf_{\rho\to0^+}\frac{m_{reg}(F,l_q,A_{\rho})}{|A_\rho|},\qquad \forall x\in\R^N\setminus{\cal N}'',\, q\in \R^{m},
\]
for every family~$(A_\rho)_{\rho>0}$ as in the statement. Finally, we accomplish the proof by choosing~${\cal N}={\cal N}'\cup {\cal N}''$.
\end{proof}

We denote by~${\cal F}={\cal F}(\alpha,C_1,C_2)$ the set of all functional~$F$ which satisfy assumptions~\eqref{assumption1}-\eqref{assumption2} with the same constants~$\alpha$,~$C_1$ and~$C_2$. In the next result, we obtain a characterization of $\Gamma$-convergence in terms of the convergence of the minima of problems with Dirichlet boundary conditions.
We like also to mention that very recently some results in this direction have been proved in \cite{Maione} in much more general geometries but with quite different techniques.

\begin{theorem}\label{T4rendic}
Let $\{F_n\}_{n\in\N}$ be a sequence of functionals in~${\cal F}$. Let~$D$ be a dense subset of~$\R^{m}$. Let~${\cal B}$ be a family of open bounded subsets of~$\R^N$ which contains a substantial family around every point~$x\in\R^N$. Assume that for each~$q\in D$ and for each~$B\in{\cal B}$ there exists $\lim_n m(F_n,q,B)$.
Then, there exists a functional~$F_\infty\in {\cal F}$ such that the sequence~$\{F_n\}_{n\in\N}$ $\Gamma$-converge to~$F_\infty$ and
\[
\lim_{n\to+\infty}m(F_n,l_q,A)=m(F_\infty,l_q,A),
\]
for every~$q\in \R^{m}$ and for every $A$  bounded domain of~$\R^N$ with Lipschitz boundary.
\end{theorem} 
\begin{proof}
The proof follows exactly the same arguments of the proof of~\cite[Theorem IV]{DalMasoModica2} so we just sketch the main issues.
We first claim that the space~${\cal F}$ can be endowed with a metric~$d$ such that~$({\cal F},d)$ is a compact metric space and a sequence~$\{F_n\}_n$ of functionals in~${\cal F}$ is convergent w.r.t. to~$d$ to some~$F\in{\cal F}$ if and only if it $\Gamma$-converges to~$F$.
Indeed, this property can be obtained following the same arguments of \cite[Proposition 1.21]{DMMampa} and taking advantage of the properties of $W_{\cX,0}^{k,\alpha}$ and of $W_{\cX}^{k,\alpha}$ for the Heisenberg group, in particular the Rellich compact injection  and the Poincar\'e inequality respectively in Lemma \ref{Lp-emb} and in Lemma  \ref{poincare}. Hence, we shall omit it.

Even if the rest of the proof follows the arguments in \cite{DalMasoModica2}, for the sake of completeness, let us recall the role of Theorem~\ref{T1rendic}. Let~$F_{k_1(n)}$ and~$F_{k_2(n)}$ two subsequences of~$F_n$ which $\Gamma$-converge respectively to some~$F_\infty'$ and to some~$F_\infty''$. We claim: $F_\infty'=F_\infty''$.
Actually, we have
\[
m_{reg}(F_\infty',l_q,B)=m_{reg}(F_\infty'',l_q,B), \qquad \forall q\in\R^{m},\, B\in{\cal B}.
\]
Theorem \ref{T1rendic} ensures that there exists a measurable set~${\cal N}\subset \R^N$, with $|{\cal N}|=0$, such that
\[
f_\infty'(x,q)=f_\infty''(x,q),\qquad \forall x\in\R^N\setminus {\cal N},\, q\in D,
\]
where $f_\infty'$ and $f_\infty''$ are the integrands of~$F_\infty'$ and respectively of~$F_\infty''$. Finally, the convexity of $f_\infty'$ and of $f_\infty''$ permits to extend the previous equality to every $q\in\R^{m}$.
\end{proof}

\section{Periodic homogenisation for degenerate functionals with $H$-affine data.}

Given the functional $F(u,A)$ defined in \eqref{F}, we now introduce for all $\varepsilon>0$ the following rescaled  functionals:
\begin{equation}
\label{RescaledFunctional}
F_\varepsilon(u,A)=\big(\rho^H_{\varepsilon}F\big)(u,A)
:=
\left\{
\begin{aligned}
\int_A
f\left(\deps(x), \nabla_{\cX}u(x)
\right)\;d\, x,&
\quad u\in W^{1,k}_{\cX}(A) \\
+\infty,&
\quad \textrm{else},
\end{aligned}
\right.
\end{equation}
and for all
$ z\in \R^N$, the following translated  functionals:
\begin{equation}
\label{translatedFunctional}
\big(\tau^H_zF\big)(u,A):=
\left\{
\begin{aligned}
\int_A f\left(z*
x,\nabla_{\cX} u(x)\right)\; d\, x,&
\quad
u\in  W_{\cX}^{k,\alpha} (A)\\
+\infty,& \quad\textrm{else}.
\end{aligned}
\right.
\end{equation}
Following the idea in \cite{DalMasoModica1}, for all fixed $q\in \R^m$, for all bounded domain $A\subset \R^N$, and with $N=2n+1$ and $m=2n$,
 we introduce the following notation
\begin{equation}
\label{mu_q}
\mu_q(A):= m(F,l_q,A)=\min\left\{
\int_Af\big(x,\nabla_{\cX} u(x)\big)\, d\;x \;\big|
u-l_q\in W^{1,\alpha}_{\cX,0} (A)
\right\},
\end{equation}
where we recall that $l_q(x)=q\cdot\pi_m(x)$ is a $H$-affine boundary data.

We next define  
$$
\tau_z^H\mu_q (A):=\mu_q\big(\tau_z^H(A)\big)=\mu_q\big(z* A\big).
$$
\begin{lemma}
\label{Identitytau_mu_q}
Given a bounded domain $A$ of $\R^N$, there holds

\begin{equation}
\label{LemmaTranslation1}
\tau_z^H\mu_q (A)=\min\left\{
(\tau^H_zF) (w,A)\;\big|
w-l_q\in W^{1,\alpha}_{\cX,0} (A)
\right\}.
\end{equation}
\end{lemma}
\begin{proof}
Note that since the functional $F$ depends only on the gradient of the function, $F(w,A)=F(w+l_q(z),A)$. Thus to prove \eqref{LemmaTranslation1} is the same of proving
\begin{equation}
\label{Claim1_translation_mu}
\tau_z^H\mu_q (A)=
\min\left\{
\big(\tau^H_zF)(w+l_q(z),A)
 \;\big|
w-l_q\in W^{1,\alpha}_{\cX,0} (A)
\right\}.
\end{equation}
In order to prove \eqref{Claim1_translation_mu}
we start looking at the right-hand side and defining
 $v(x):=w(x)+l_q(z)$. Since $w$ and $v$ differ only by a constant, obviously
\begin{equation}
\label{Jova1}
\begin{aligned}
&
\min
\left\{
\big(\tau^H_zF)(w+l_q(z),A)
 \;\big|
w-l_q\in W^{1,\alpha}_{\cX,0} (A)
\right\}\\
=&
\min
\left\{
\int_A
f\big(z* x,\nabla_{\cX}
 \big[w(x)+l_q(z)
\big]\big)\, d\;x 
 \;\big|
w-l_q\in W^{1,\alpha}_{\cX,0} (A)
\right\}
\\
=&
\min
\left\{
\int_A
f\big(z* x,\nabla_{\cX} v(x)
\big)\, d\;x 
 \;\big|
v-L_z(l_q)\in W^{1,\alpha}_{\cX,0} (A)
\right\},
\end{aligned}
\end{equation}
where we recall that $L_z(l_q)(x):=l_q(z* x)$ by definition of translated function.\\
Now we consider the following change of variables $y=z* x$ (equivalently $x=z^{-1}* y$ where $z^{-1}$ is the inverse element w.r.t. the group law $*$).

An easy computation shows that the Jacobian of the change of variables is exactly the matrix $\sigma_{Ext}$ defined in \eqref{MatrixHeisenberg_extended}. 
Then property \eqref{DeterminantMatrixExt} tells that $|\det J|=1$. Since $X_i$ are defined as left-invariant vector fields 
for all $i=1,\dots m$
(see \eqref{LeftInvariant_vectorFields}) we also know  that
$$
\nabla_{\cX} v (x)=\nabla_{\cX} \big(L_{z^{-1}} (v)\big)(x)=\nabla_{\cX} v(z^{-1}* x).
$$
Moreover $x\in A$ if and only if $y\in z* A$ and 
$$
v-L_z(l_q)\in W^{1,\alpha}_{\cX,0} (A) 
\quad
\textrm{ if and only if }
\quad
L_{z^{-1}}(v)-l_q\in  W^{1,\alpha}_{\cX,0} (z* A),
$$ 
where 
$L_{z^{-1}}(v)(x):=v(z^{-1}* x)$:
in fact  on $\partial (z* A)$ we have $v(z^{-1}* y)=v(x)=l_q(z* x)=l_q(z* z^{-1}* y)=l_q(y)$.
Then in the new variables $y=z* x$ we have
\begin{equation}
\label{Jova2}
\begin{aligned}
&
\min
\left\{
\int_A
f\big(z* x,\nabla_{\cX} v(x)
\big)\, d\;x 
 \;\big|
v-L_z(l_q)\in W^{1,\alpha}_{\cX,0} (A)
\right\}\\
=&
\min
\left\{
\int_{z* A}
f\big(y,\nabla_{\cX} v(z^{-1}* y)
\big)\, d\;y
 \;\big|
 L_{z^{-1}}(v)-l_q
 \in W^{1,\alpha}_{\cX,0} (z* A)
\right\}.
\end{aligned}
\end{equation}
To conclude we now define 
$
u:=L_{z^{-1}}
(v)$. 
Using again the property of left-invariant vector fields, we have 
$
\nabla_{\cX}
v (z^{-1}* y)
=\nabla_{\cX} u(z* z^{-1}* y)= \nabla_{\cX} u(y),
$
 then
\begin{equation}
\label{Jova3}
\begin{aligned}&
\min
\left\{
\int_{z* A}
f\big(y,\nabla_{\cX} v(z^{-1}* y)
\big)\, d\;y
 \;\big|
 L_{z^{-1}}(v)-l_q
 \in W^{1,\alpha}_{\cX,0} (z* A)
\right\}\\
=&
\min
\left\{
\int_{z* A}
f\big(y,\nabla_{\cX} u(y)
\big)\, d\;y
 \;\big|
u-l_q
 \in W^{1,\alpha}_{\cX,0} (z* A)
\right\}
=\tau_z^H\mu_q (A).
\end{aligned}
\end{equation}
The chains of identities in \eqref{Jova1}, \eqref{Jova2} and \eqref{Jova3} give identity \eqref{Claim1_translation_mu} and conclude the proof.
\end{proof}

The following result is an immediate consequence of the previous lemma in the case of $H$-periodic functionals.
\begin{lemma}
\label{Proposition_periodic}
Assume \eqref{assumption-Periodic}, then, for all bounded  domains $A\subset \R^N$  and 
for all $q\in \R^{m}$ and $z\in \Z^N$
$$
\tau_z^H\mu_q (A)=\mu_q (A).
$$
\end{lemma}

In the following lemma we show how the assumptions on the integrand $f(x,q) $ are inherited by $\mu_q(A)$.
\begin{lemma}
\label{Boundness_mu_q}
Let $A$ be a bounded domain in $\R^N$ with Lipschitz boundary, $q\in \R^{m}$ and $\mu_q(A)$ defined in \eqref{mu_q} and let  $f:\R^N\times \R^{m}\to \R$ be measurable. 
\begin{enumerate}
\item[(i)] If $f$ satisfies assumption \eqref{assumption2}, we have
$$
C_1|q|^\alpha|A|\leq \mu_q(A)\leq C_2\big(|q|^\alpha+1\big) |A|,
$$
where $C_1,C_2$ and $\alpha$ are the same constants given in  \eqref{assumption2}.
\item[(ii)] If $f$ satisfies  assumption \eqref{assumption1}, we have
 for all $q_1,q_2\in \R^{m}$ and for all $\lambda\in (0,1)$
$$
\mu_{\lambda q_1+(1-\lambda)q_2}(A)\leq \lambda\mu_{q_1}+(1-\lambda)\mu_{q_2}.
$$
\end{enumerate}
\end{lemma}
\begin{proof}
Using that $l_q= q\cdot \pi_{m}$ is admissible for the minimum defining $\mu_q(A)$ and that 
$ \nabla_{\cX}
l_q(x)=q$, we get
$$
\mu_q(A)\leq \int_Af\big(x, \nabla_{\cX} l_q(x)\big)\; d\,x
=
\int_Af\big(x,  q\big)\; d\,x
\leq 
C_2\big(|q|^\alpha+1\big)|A|.
$$
Moreover
\begin{align*}
\mu_q(A)&=
\min\left\{
\int_Af\big(x,\nabla_{\cX} u(x)\big)\, d\;x \;\big|
u-l_q\in W^{1,\alpha}_{\cX,0} (A)
\right\}\\
&\geq C_1
\min\left\{
\int_A|\nabla_{\cX} u(x)|^\alpha\, d\;x \;\big|
u-l_q\in W^{1,\alpha}_{\cX,0} (A)
\right\}\\
&=C_1\int_A|\nabla_{\cX} l_q(x)|^\alpha
\, d\;x=
C_1|q|^\alpha|A|,
\end{align*}
where for the last identity we  use Lemma \ref{Lemma2} for the convex function $f(x,q)=|q|^{\alpha}$, which tells that the minimisers are the $H$-affine functions, whenever $f$ does not depend on $x$. 

It remains to prove (ii). To this end it is enough to remark that for all functions $u_1$ and $u_2$ which are admissible respectively  for $\mu_{q_1}$ and $\mu_{q_2}$, then $\overline{u}:=\lambda u_1+(1-\lambda)u_2$ is admissible for $\mu_{\lambda q_1+(1-\lambda)q_2}$, which implies
\begin{align*}
\mu_{\lambda q_1+(1-\lambda)q_2}(A)&\leq 
\int_A f(x,\nabla_{\cX}\overline{u}(x))d\,x\\
&
=\int_A f\big(x,
\lambda \nabla_{\cX}u_1(x)+(1-\lambda)\nabla_{\cX}u_2(x)
\big)
\;d\,x\\
&\leq 
\lambda
\int_A f(x,
 \nabla_{\cX}u_1(x))\;d\,x
 +
(1-\lambda)
\int_A f(x,
 \nabla_{\cX}u_2(x))\;d\,x.
\end{align*}
Taking the minimum over all admissible $u_1$ and $u_2$, we get property (ii).
\end{proof}

To prove the convergence of the functional $F_\varepsilon(u,A)$ as $\varepsilon \to 0^+$, we need to show now a sort of  Akcoglu-Krengel  type result (see \cite{AK}) for periodic functionals, adapted to the anisotropic structure of the Heisenberg group. 
In \cite{Hafsa}  the authors prove a very interesting Akcoglu-Krengel  type result for general metric measure spaces.  We need to mention that unfortunately the result therein does not apply to our case. In fact it is quite easy to show that the Heisenberg group endowed with the Carnot-Carath\'eodory metric (or also with the homogeneous metric) and the Lebesgue measure is a $\big(G,\{\delta_t\}_{t>0}\big)$-metric measure space 
where $G$ is the subgroup of homeomorphisms on the Heisenberg group defined by the left-translations w.r.t. an element in $\Z^N$. Nevertheless one can also show that in general that space is not ``meashable''  according to the definition introduced in \cite{Hafsa}.
  We give a self-contained proof which can be later adapted to the stochastic case (which will be a topic in a forthcoming paper, see Section 6).
  
  We now recall that, defining for all $ t>0$, 
 $
Q^t=\delta_t(Q)
$,
  we know that $|\delta_t(Q)|=t^\mathcal{Q}|Q|$ (see \eqref{Volume}), where $\mathcal{Q}$ is the homogeneous dimension, then in $\dH^n$ in particular $\mathcal{Q}=2n+2=N+1$.

The next lemma tells that, as $t\to+\infty$, we can reduce to take the limits only over integer subsequences.
\begin{lemma}
\label{Reduction to integers}
Assume that the limit exists for integer sequences, i.e.  for $h\in \N$,
$$
\lim_{h\to\infty}\frac{\mu_q(Q^h)}{|Q^h|}=:C.
$$ 
 Then for all sequences $\{t_k\}\subset \R$ with $t_k\to \infty$ 
it holds
$$
\lim_{k\to\infty}\frac{\mu_q(Q^{t_k})}{|Q^{t_k}|}=C.
$$
\end{lemma}
\begin{proof}
For $t>0$ we define
$$
e_t:= \frac{\mu_q(Q^{t})}{|Q^{t}|}.
$$
Fix $\varepsilon>0$ and choose $N$  large enough that
$|C-e_h|<\varepsilon$ for $h\ge N.$

Denote by $C^+:=\limsup_{k\to\infty}e_{t_k},$ and 
$C^-:=\liminf_{k\to\infty}e_{t_k},$ which are both finite by Lemma \ref{Boundness_mu_q}.
We can find $k$ such that 
$$
e_{t_k}\ge C^+-\varepsilon
\quad
\textrm{and}
 \quad t_k>N.
$$
Define $N_k:=[ t_k]\ge N,$  (where by $[\cdot]$ we indicate the integer part of a real number)
and let $u_k$ be a function with $H$-affine boundary conditions on $Q^{N_k}$ such that $F(u_k, Q^{N_k})=\mu_q(Q^{N_k}).$ We extend
$u_k$ to $Q^{t_k}$ by letting it equal to the boundary condition on $Q^{t_k}\setminus Q^{N_k}$,
i.e. $\widetilde{u}_k:\ \R^N\to \R$ given by
$$
 \widetilde{u}_k(x):=
\left\{\begin{array}{ll} u_k(x),&\ {\rm if}\ x\in Q^{N_k},\\ 
l_q(x),&\ {\rm else,}
\end{array}\right.
$$
whose restriction to $Q^{t_k}$ is an admissible function for  $\mu_q(Q^{t_k}).$
Note that
$$
f(x,\nabla_{\cX}\widetilde u_k)=f(x,\nabla_{\cX}l_q)=f(x,q)\le C_2(|q|^\alpha+1)\quad \textrm{on} 
\quad
Q^{t_k}\setminus Q^{N_k},
$$
 hence
$$
F(\widetilde u_k,Q^{t_k})= \int_{Q^{N_k}}f(x,\nabla_{\cX}u_k)dx+\int_{Q^{t_k}\setminus Q^{N_k}}\!\!\!\!\!\!\!\!\!\!\!\!
f(x,q)dx\le  F(u_k,Q^{N_k}) +C| Q^{t_k}\setminus Q^{N_k}|, 
$$ where the constant depends on $q$ and $\alpha.$
 Since $\widetilde u_k$ is admissible for $Q^{t_k},$ so $\mu_q(Q^{t_k})\le F(\widetilde u_k,Q^{t_k}),$
we estimate
\begin{eqnarray*}
C^+&\le& e_{t_k}+\varepsilon\le \frac{F(\widetilde u_k,Q^{t_k})}{|Q^{t_k}|}+\varepsilon\le \frac{F(u_k,Q^{N_k})}{|Q^{t_k}|} +\varepsilon+ C\frac{| Q^{t_k}\setminus Q^{N_k}| }{|Q^{t_k}|} \\&=&\varepsilon+\frac{\mu_q(Q^{N_k})}{|Q^{t_k}|}+C\frac{| Q^{t_k}\setminus Q^{N_k}| }{|Q^{t_k}|}=e_{N_k}\frac{|Q^{N_k}|}{|Q^{t_k}|}+\varepsilon+ C\frac{| Q^{t_k}\setminus Q^{N_k}| }
{|Q^{t_k}|}.
\end{eqnarray*}
Note that $$
\lim_{k\to \infty} \frac{|Q^{N_k}|}{|Q^{t_k}|}=1\quad
\textrm{and}
\quad \lim_{k\to\infty} \frac{| Q^{t_k}\setminus Q^{N_k}| }{|Q^{t_k}|}=0,
$$
so, by choosing, if necessary, $N$ larger, we can make the right hand side $\le C+3\varepsilon,$ thus, as $\varepsilon$ was arbitrary,  we have shown $C^+\le C.$

For the opposite inequality, we use estimates similar to what we did before:  we can find infinitely many $k$ such that 
$$
e_{t_k}\le C^-+\varepsilon
\quad
\textrm{and}
 \quad
 t_k>{ N}.$$
Therefore we take $N_k=[t_k]+1 $ and let $u_k$ be a function with $H$-affine boundary condition $l_q$  on $Q^{t_k}$ such that $F(u_k, Q^{t_k})=\mu_q(Q^{t_k}).$ We extend
$u_k$ to a function  $\widetilde u_k$ on $Q^{N_k}$ which is admissible for $\mu_q(Q^{N_k})$ and equals $l_q$ on
$Q^{N_k}\setminus Q^{t_k}.$ 
Arguing as before we get
$$
F(\widetilde u_k,Q^{N_k})\le F(u_k,Q^{t_k})+C| Q^{N_k}\setminus Q^{t_k}|. 
$$ Then
\begin{eqnarray*}
C &\le &  e_{N_k}+\varepsilon\le \frac{F(\widetilde u_k,Q^{N_k})}{|Q^{N_k}|}+\varepsilon\le \frac{F(u_k,Q^{t_k})}{|Q^{N_k}|} +\varepsilon+ C\frac{| Q^{N_k}\setminus Q^{t_k}| }{|Q^{N_k}|}\\ &=&\varepsilon+\frac{\mu_q(Q^{t_k})}{|Q^{N_k}|}+C\frac{| Q^{N_k}\setminus Q^{t_k}| }{|Q^{N_k}|}=e_{t_k}\frac{|Q^{t_k}|}{|Q^{N_k}|}+\varepsilon+ C\frac{| Q^{N_k}\setminus Q^{t_k}| }
{|Q^{N_k}|}\\ &\le&
(C^-+\varepsilon)+(C^-+\varepsilon)\left(\frac{|Q^{t_k}|}{|Q^{N_k}|}-1 \right)+\varepsilon +C\frac{| Q^{N_k}\setminus Q^{t_k}| }
{|Q^{N_k}|}.
\end{eqnarray*}
By choosing, if necessary, $N$ larger, we can make the right hand side smaller than $ C^-+(2C^-+3)\varepsilon,$ thus we have shown $C\le C^-,$ but as $C^-\le C^+,$ we have $C^-=C=C^+.$

\end{proof}

 We denote $\N^*$ the set of natural numbers excluding 0. We next prove an Akcoglu-Krengel  type result.

\begin{teo}
\label{LimiteImportante}
Let consider the (semiopen) unit cell $Q=[-1,1)^N$ and let $q\in \R^{m}$ and $\mu_q$ be defined in \eqref{mu_q}. Assume  that $f$ is measurable and satisfies \eqref{assumption2} and \eqref{assumption-Periodic}, then 
$$
\lim_{k\to +\infty}\frac{\mu_q\big(Q^k\big)}{|Q^k|}=C_q,
$$
where $C_q$ is the non-negative constant given by
$$
C_q=\inf_{k\in \N^*}\frac{\mu_q\big(Q^k\big)}{|Q^k|}
.$$
\end{teo}
\begin{proof}
Note that, by Lemma \ref{Boundness_mu_q} - (i),
\begin{equation}
\label{VirginRadio}
C_1|q|^\alpha\leq
\frac{\mu_q\big(Q^k\big)}{|Q^k|}
\leq C_2\big(|q|^\alpha+1\big),
\end{equation}
so in particular $C_q\geq 0$.\\

{\bf Step 1.} Since $C_q$ is defined as  infumum over $\N^*$, trivially 
$
\frac{\mu_q\big(Q^k\big)}{|Q^k|}\geq C_q
$ for all $k\in \N^*$,
which implies
$$
\liminf_{k\to +\infty}\frac{\mu_q\big(Q^k\big)}{|Q^k|}\geq C_q.
$$

{\bf Step 2.} We next show the limsup estimate.
Using the definition of infimum for $C_q$, for all $\rho>0$, and the definition of $\mu_q$,
 then there exists $k_\rho\in \N^*$  and $u_\rho\in W^{1,\alpha}_{\cX}(Q^{k_\rho})
 \cap C^\infty  (Q^{k_\rho})$ 
 and such that 
\begin{align*}
&\frac{\mu_q(Q^{k_\rho})}
{|Q^{k_\rho}|}
\leq C_q+\frac{\rho}{2},\\
&
\frac{F(u_\rho,Q^{k_\rho})}
{|Q^{k_\rho}|}
\leq
\frac{\mu_q(Q^{k_\rho})}
{|Q^{k_\rho}|}+\frac{\rho}{2},
\end{align*}
where $F$ is the functional defined in \eqref{F}. This sums up as follows
\begin{equation}
\label{Sala}
\frac{F(u_\rho,Q^{k_\rho})}{|Q^{k_\rho}|}\leq
C_q+\rho.
\end{equation}
Recall the definition of 
$
\tau_k$ given in
\eqref{TranslazioneImportante};
we use such translations  to extend $u_\rho$ to the whole $\R^N$ by translating periodically the gradient. 
More precisely, 
let us introduce 
$$j_\rho:=
\delta_{k_\rho}(j)
\quad
\textrm{and}\quad
Q^\rho_j:= \tau_{j_\rho}\big(Q^{k_\rho}\big),
\quad 
\forall \; j\in \Z^N.
$$
Using that $\R^N=\bigcup_{j\in \Z^N} Q^\rho_j$, we can define
$$
U_\rho(x):=\sum_{j\in \Z^N}
\bigg(
q\cdot \pi_{m}(j_\rho  )+ u_\rho \big(\tau_{- j_\rho}(x)\big)
\bigg) \Uno_{ Q^\rho_j},
$$
where by $\Uno_A$ we indicate the 
characteristic function of the set $A$;
recall also that $\tau_{k}^{-1}=\tau_{-k}$.
The function $U_\rho$ is well-defined since 
$Q^\rho_j$ are all disjoint. 
We can easily check that $U_\rho$ is continuous on $\R^N$:  
in fact, for $x\in  Q^\rho_j$, then
$
U_\rho (x)=q\cdot \pi_{m}(j_\rho)+
u_\rho(\tau_{-j_\rho}(x))
$ and, whenever $x\in
\partial Q^\rho_j$, we have  $\tau_{-j_\rho}(x)\in \partial Q^{k_\rho}$ which implies
$$
U_\rho (x)=q\cdot \pi_{m}(j_\rho ) +q\cdot \pi_{m}\big(\tau_{-j_\rho}(x)\big)=
q\cdot \pi_{m}(j_\rho ) +q\cdot
\big(\pi_{m}(-j_\rho)+\pi_{m}(x)\big)=
q\cdot\pi_{m}(x),
$$
which does not anymore depend on $j$.
The continuity of $U_\rho$ on $\R^N$, together with the  fact that $U_\rho\in W^{1,\alpha}_{\cX}(Q^\rho_j)$, imply that
$U_\rho\in W^{1,\alpha}_{\cX,\textrm{loc}} \big(\R^N\big)$. \\

We next introduce the following two objects:
\begin{align*}
& S^\rho_k:=\bigg\{
j\in \Z^N|\, Q_{j}^{\rho}\subset Q^k=\delta_{k}(Q)
\bigg\},\\
&
\widehat{S}^\rho_k:=\bigcup_{j\in S^\rho_k}Q_j^\rho,
\end{align*}
and we construct a new function $v_\rho$, which is admissible for $\mu_q(Q^k)$, as
$$
v_\rho(x):=
\left\{
\begin{aligned}
&U_\rho(x),
\quad\quad\quad x\in \widehat{S}^\rho_k\\
&q\cdot\pi_{m}(x),
\quad\; x\in  Q^k\backslash\widehat{S}^\rho_k.
\end{aligned}
\right.
$$
By definition $v_\rho-l_q\in W^{1,\alpha}_{\cX,0}\big(Q^k\big)$, and
\begin{align*}
F(v_\rho, Q^k)&=
\int_{Q^k}
f(x,\nabla_{\cX}v_\rho(x))\,d\,x
\\
&=
\int_{ \widehat{S}^\rho_k}
f(x,\nabla_{\cX}v_\rho(x))\,d\,x+
\int_{Q^k\backslash  \widehat{S}^\rho_k}
f(x,\nabla_{\cX}v_\rho(x))\,d\,x
\end{align*}
First we compute
\begin{equation}
\label{Sabato1}
\int_{ \widehat{S}^\rho_k}
f(x,\nabla_{\cX}v_\rho(x))\,d\,x
=
\int_{ \widehat{S}^\rho_k}
f(x,\nabla_{\cX}U_\rho(x))\,d\,x=
\sum_{j\in S^\rho_k} 
\int_{ Q_j^\rho}
f(x,\nabla_{\cX}U_\rho(x))\,d\,x,
\end{equation}
where we have used that $Q_j^\rho$ are disjoint. If $x\in Q_j^\rho$, then
$U_\rho(x)
=q\cdot \pi_{m}(j_\rho)+u_\rho\big(\tau_{-j_\rho}(x)\big)$. By using that the vector fields are left invariant, we get
$$
\nabla_{\cX}U_\rho(x)=
\nabla_{\cX} \bigg(u_\rho\big(\tau_{-j_\rho}(x)\big)\bigg)=
\nabla_{\cX}u_\rho\big(\tau_{-j_\rho}(x)\big).
$$
Thus, by using the change of variables $y=\tau_{-j_\rho}(x)$ and recalling that the determinant of the Jacobian is 1  (see \eqref{DeterminantMatrixExt}), we get the following chain of identities:
\begin{equation}
\label{Sabato2}
\begin{aligned}
&\sum_{j\in S^\rho_k} 
\int_{ Q_j^\rho}
f(x,\nabla_{\cX}U_\rho(x))\,d\,x
=
\sum_{j\in S^\rho_k} 
\int_{ Q_j^\rho}
f\big(x,\nabla_{\cX}u_\rho\big(\tau_{-j_\rho}(x)
\big)\big)\,d\,x\\
=&
\sum_{j\in S^\rho_k} 
\int_{ Q^{k_\rho}}
f\big(\tau_{j_\rho}(y),\nabla_{\cX}u_\rho\big(y)
\big)\,d\,y
=\sum_{j\in S^\rho_k} 
\int_{ Q^{k_\rho}}
f\big(y,\nabla_{\cX}u_\rho\big(y)
\big)\,d\,y,
\end{aligned}
\end{equation}
where in the last identity above we have used the periodicity assumption on $f$ (see assumption \eqref{assumption-Periodic}). The integrals in the last term of \eqref{Sabato2} do not depend anymore on $j$, then 
\begin{equation}
\label{Sabato3}
\begin{aligned}
\sum_{j\in S^\rho_k} 
\int_{ Q^{k_\rho}}
f\big(y,\nabla_{\cX}u_\rho\big(y)
\big)\,d\,y&=
\textrm{card} \big(S^\rho_k\big)
\int_{ Q^{k_\rho}}
f\big(y,\nabla_{\cX}u_\rho\big(y)
\big)\,d\,y\\
&\leq
\textrm{card} \big(S^\rho_k\big) \big(C_q+\rho) |Q^{k_\rho}|,
\end{aligned}
\end{equation}
where the last inequality follows from \eqref{Sala}.\\
Put together \eqref{Sabato1},\eqref{Sabato2} and \eqref{Sabato3}, we get the following estimate:
\begin{equation}
\label{Sabato4}
\int_{ \widehat{S}^\rho_k}
f(x,\nabla_{\cX}v_\rho(x))\,d\,x
\leq
\textrm{card} \big(S^\rho_k\big) \big(C_q+\rho) |Q^{k_\rho}|.
\end{equation}

It remains to estimate the integral on the complementary of $\widehat{S}^\rho_k$ by using that $v_\rho(x)=q\cdot \pi_{m}(x)$ for all $x\in Q^{k}\backslash \widehat{S}^\rho_k$ by definition, hence
\begin{equation}
\label{Sabato5}
\int_{Q^k\backslash  \widehat{S}^\rho_k}
f(x,\nabla_{\cX}v_\rho(x))\,d\,x
\leq 
\int_{Q^k\backslash  \widehat{S}^\rho_k}
f(x,q)\,d\,x\leq
C_2(|q|^\alpha+1) |Q^k\backslash  \widehat{S}^\rho_k|.
\end{equation}
Estimates \eqref{Sabato4} and  \eqref{Sabato5}, together with the fact that $v_\rho$ is admissible for $\mu_q(Q^k)$, give
\begin{equation}
\label{QuasiFine} 
\begin{aligned}
\frac{\mu_q(Q^{k})}{|Q^k|}\leq 
\frac{F(v_\rho,Q^k)}{|Q^k|}\leq
\textrm{card} \big(S^\rho_k\big) \big(C_q+\rho) \frac{|Q^{k_\rho}|}{|Q^k|}+
C_2(|q|^\alpha+1) \frac{|Q^k\backslash  \widehat{S}^\rho_k|}{|Q^k|},
\end{aligned}
\end{equation} 
where in the last inequality we have used that $ \widehat{S}^\rho_k\subset Q^k$ and
$| \widehat{S}^\rho_k|=\textrm{card} \big(S^\rho_k\big) |Q^{k_\rho}|$, which together  imply
$\frac{\textrm{card} \big(S^\rho_k\big)\,|Q^{k_\rho}|}{|Q^k|}\leq 1$.

To conclude we  claim that the following limit holds true:
\begin{equation}
\label{claimLimit}
\lim_{k\to+\infty}\frac{|Q^k\backslash  \widehat{S}^\rho_k|}{|Q^k|}=0.
\end{equation}
Then, by simply taking the limsup as $k\to +\infty$ in the inequality \eqref{QuasiFine} and using claim 
 \eqref{claimLimit}, we get
$$
\limsup_{k\to +\infty}\frac{\mu_q(Q^{k})}{|Q^k|}\leq 
 \big(C_q+\rho),
$$
which conclude the proof as $\rho\to 0^+$.\\

It remains only now to prove claim
 \eqref{claimLimit}.
 By a simple rescaling we can actually show that this limit is the same as the one shown in the proof of Lemma 2.21 in \cite{Franchi2}. In fact, set $\varepsilon=\frac{1}{k}$, then by using the properties of dilations (Lemma \ref{DilationsProperties})
 $$
 Q^{\varepsilon}=\delta_{\varepsilon}(Q)=
 \delta_{\frac{1}{k_\rho}}\bigg(
  \delta_{k_\rho}\big(
  \delta_\varepsilon(Q)
  \big)
 \bigg)=
 \delta_{\frac{\varepsilon}{k_\rho}}(Q^{k_\rho}).
 $$
Set $\widetilde{Q}:=\delta_\varepsilon(Q^{k_\rho})=
 \delta_{\frac{\varepsilon}{k_\rho}}(Q^{\frac{1}{\varepsilon}})
$,  by using the properties of dilations and left-translations one can easily check that
$$
\tau_{\delta_{\varepsilon}(j)}
(Q^\varepsilon)\subset \widetilde{Q}
\quad
\Longleftrightarrow
\quad
\tau_{\delta_{k_\rho}(j)}
(Q^{k_\rho}).
$$
Thus by using the limit proved in  \cite{Franchi2} we conclude the proof.
\end{proof}

We define $f_0:\R^{m}\to\R$
as
\begin{equation}
\label{f_effective}
f_0(q):=C_q,
\end{equation}
where $C_q$ is the limit proved in Theorem \ref{LimiteImportante}.\\
From Lemma \ref{Boundness_mu_q}, one can show that $f_0$ keeps the properties of $f$ simply by passing to the limit as $k\to +\infty$. More precisely
\begin{lemma}
\label{Properties_f_0}
Given $f:\R^N\times \R^{m}\to \R$  measurable,
the following properties hold:
\begin{enumerate}
\item[(i)] if assumption \eqref{assumption2} is satisfied, then 
$$
C_1|q|^\alpha
\leq f_0(q)\leq C_2 \big(|q|^\alpha+1\big),
$$
where $C_1,C_2$ and $\alpha$ are the same constants given in  \eqref{assumption2}.
\item[(ii)]  if  assumption \eqref{assumption1} is satisfied,
then for all $q_1,q_2\in \R^{m}$ 
$$
f_0(\lambda q_1+(1-\lambda)q_2)
\leq \lambda f_0({q_1})+(1-\lambda)f_0({q_2}),
\quad \lambda\in (0,1).
$$
\end{enumerate}
\end{lemma}

We now prove the main result of the paper.
\begin{teo}
\label{MainTeoHomog}
Given a  bounded  domain $A\subset \R^N$ with Lipschitz boundary, $u:A\to \R$ and  the functional $F(u,A)$ defined in \eqref{F}. Let us assume that \eqref{assumption1},  \eqref{assumption2} and  \eqref{assumption-Periodic}  hold true, and  $u_0(x)=q\cdot \pi_m(x)+a$ for some $q\in \R^m$  and  $a\in \R$.
Define the rescaled functionals $F_{\varepsilon}$ introduced in \eqref{RescaledFunctional} and let us consider the corresponding minimisation problems for $u-u_0\in W^{1,\alpha}_{\cX,0}(A)$ (see \eqref{min})
then 
$$
\lim_{\varepsilon\to 0^+} m(F_{\varepsilon},u_0,A)
=m(F_{\infty},u_0,A),
$$ where the limit functional $F_{\infty}$ can be characterised as 
\begin{equation*}
F_{\infty}(u,A):=
\left\{\begin{aligned}
\int_A f_0\left(\nabla_{\cX} u\right)\; d\, x,&
\quad
u\in  W_{\cX}^{1,\alpha} (A),\\
+\infty,& \quad\textrm{else},
\end{aligned}
\right.
\end{equation*}
and  $f_0:\R^m\to \R$ defined  as $f_0(q)=C_q$ with  $C_q$ constant given in Theorem \ref{LimiteImportante}.

Moreover the limit function $f_0$ is still measurable, convex and satisfies the same growth condition \eqref{assumption2} satisfied by $f$.
\end{teo}

\begin{proof}
Applying Theorem \ref{T4rendic} we deduce that $F_{\varepsilon}$ $\Gamma$-converge to some functional $F_{\infty}$.
Let us now prove that the limit functional $F_{\infty}$ can be identified as the integral functional associated to $f_0$ given in \eqref{f_effective}. 
Choose as substantial family $A_\rho:=[-\rho,\rho]^N$, and
fix $t=\frac{1}{\varepsilon}$,  at the moment let us  assume the following claim:
\begin{equation}
\label{ULTIMO}
\mu_q(\delta_t(A_\rho))=t^\mathcal{Q}
m(F_{\varepsilon}, l_q,A_\rho),
\end{equation}
with $\mathcal{Q}$ homogeneous dimension. 

By using that all the  previous results in Section 5 can be obtained replacing $Q$ with the cube $A_\rho$,  we have
\begin{align*}
\frac{m(F_{\infty},l_q,A_\rho)}{|A_\rho|}&=
\lim_{\varepsilon\to 0^+} 
\frac{m(F_{\varepsilon},l_q,A_\rho)}{|A_\rho|} 
=\lim_{t\to +\infty}
\frac{1}{|A_\rho|}
 \frac{\mu_q(\delta_t(A_\rho))}{t^\mathcal{Q}}\\
 &=
\lim_{t\to +\infty}
 \frac{\mu_q\big(\delta_t(A_\rho)\big)}{|\delta_t(A_\rho)|}
=C_q,
\quad \forall\, \rho>0.
\end{align*}
By Theorem \ref{T1rendic}, passing to the limit as $\rho \to 0^+$,  we conclude $f_0(q)=C_q$.

It only remains to check claim \eqref{ULTIMO}.
At this purpose, we use the change of variables $y=\delta_{1/t}(x)$; hence recalling definition \eqref{mu_q}
and using that $\delta_t(u)$ is the scaled function defined as $\delta_t(u)(x)=u(\delta_t(x))$, we have
\begin{equation}
\label{finalissima}
\begin{aligned}
\mu_q(\delta_t(A_\rho))
&=\min\left\{
\int_{\delta_t(A_\rho)}f(x,\nabla_{
\cX} u(x))\,d\,x\;\big|\; u-l_q\in W^{1,\alpha}_{\cX,0}(\delta_t(A_\rho))
\right\}\\
&=t^{\mathcal{Q}}
\min\left\{
\int_{A_\rho}f(\delta_t(y),\nabla_{
\cX} u(\delta_t(y)))\,d\,y\;\big|\; \delta_t(u)-l_{tq}\in W^{1,\alpha}_{\cX,0}(A_\rho)
\right\},
\end{aligned}
\end{equation}
by simply using that $l_q(\delta_t(y))=q\cdot \pi_m(\delta_t(y))=q\cdot t\, \pi_m(y)=(tq)\cdot \pi_m(y)$ and \eqref{Volume}. Defining the function $w:=\frac{1}{t}\delta_t(u)$ and  using   Lemma \ref{L:xjhom}, we have that
$$
\nabla_{\cX}w(y)=\frac{1}{t} \;\nabla_{\cX}(\delta_t(u))(y)=\frac{1}{t}\; t\; \nabla_{\cX}u(\delta_t(y))=\nabla_{\cX}u(\delta_t(y)).
$$
Moreover informally we have  that, for $y\in \partial A_\rho$,  $\frac{1}{t}u(\delta_t(y))=\frac{1}{t}l_q(\delta_t(y))=l_q(y)$. Hence 
 \eqref{finalissima}  gives
 \begin{align*}
\mu_q(\delta_t(A_\rho))
&=t^{\mathcal{Q}}
\min\left\{
\int_{A_\rho}f(\delta_t(y),\nabla_{
\cX} w(y))\,d\,y\;\big|\; w-l_{q}\in W^{1,\alpha}_{\cX,0}(A_\rho)
\right\}\\
&=t^\mathcal{Q}
m(F_{1/t}, l_q,A_\rho),
\end{align*}
 which proves claim \eqref{ULTIMO}.

The properties for the limit function $f_0$ are proved in Lemma \ref{Properties_f_0}.
\end{proof}

\section{Applications and generalizations}
We conclude listing further directions in which we are presently 
working, for some of  which we obtained already some partial results.

\subsection{Homogenisation for  functionals associated to Carnot groups
and the subelliptic $p$-Laplacian.}

As mentioned in the introduction all the proofs never use the specific structure of the Heisenberg group but they instead  use properties true for all Carnot groups. So all the results apply without any modification to the general case of Carnot groups.\\

As it is well-known by Euler-Lagrange equations, we can connect minima of functionals  to solutions of PDEs. Whenever uniqueness holds this correspondence is one-to-one. Then our results can be used to study homogenisation for several subelliptic PDEs and in particular for the subelliptic $p$-Laplacian, which is defined, for  $1<p<+\infty$, as 
$$
\textrm{div}_{\cX}\left(\big<\mathcal{A}\nabla_{\cX}u,\nabla_{\cX}u\big>^{\frac{p-2}{2}}\mathcal{A}\nabla_{\cX}u
\right)
=0,
$$
where $\mathcal{A}(x)$ is a $m\times m$ symmetric matrix satisfying the usual  ellipticity condition.
Equations of this form have been studies by many authors, see e.g. \cite{ricciotti} and references therein.
The functional associated to the $p$-Laplacian is
$$
F_p(u,A)=
\left\{
\begin{aligned}
\int_A\big|\mathcal{A}\big(\delta_{\frac{1}{\varepsilon}}(x)\big)\nabla_{\cX}u(x)\big|^p\, d\, x,& \quad 
u\in W^{1,p}_{\cX}(A)\\
+\infty,& \quad\textrm{else}.
\end{aligned}
\right.
$$
Note that $ F_p$ satisfies all our conditions for all $1<p<+\infty$. Then we can apply Theorem \ref{MainTeoHomog}. It remains now to show that the limit functional has still the structure of a functional associated to a subelliptic $p$-Laplacian equation (work in preparation).

\subsection{Stochastic functionals.}
Another generalisation  is the case of 
 random functionals, i.e. integral functionals of the form
$$
u\mapsto
F_\varepsilon(u, A)=\int_{A}f\left(\delta_{\frac{1}{\varepsilon}}(x),\omega,\nabla  u(x)\right) dx,
$$
where $\omega$ belongs to a probability space and 
the integrand $f(x,\omega,p)$ is stationary and ergodic  with respect to left translations. 
For a precise definition of stationary ergodic in the setting of Carnot groups we refer to
 \cite{DDMM}, where the authors  prove an homogenisation result for stochastic Hamilton-Jacobi equations.
 The general stationary ergodic case
 will be treated in a forthcoming  paper, but we sketch here a proof for the simpler situation of short correlated random variables. 

More precisely, we assume that the random integrand $f(x,\omega, p)$ satisfies  \eqref{assumption1} and \eqref{assumption2} uniformly in $\omega$  and \eqref{assumption-Periodic} {\em in law}, i.e. the random integrand and its translations are not equal, but have the same law as random variables. In addition, we require that there exits a constant $C>0$ such that $f(x,\omega,p)$ and $f(y,\omega,p)$ are  independent, if $d_h(x,y)>C,$  where by $d_h(x,y)$ we indicate the homogeneous distance in Carnot groups, i.e. for example in  1-dimensional Heisenberg $d_h(x,y)=|y^{-1}*x|_h$ where $|x|_h:=\big((x_1^2+x_2^2)^2+x_{3}^2\big)^{1/4}$.
Note that this is different from being short correlated in the Euclidean distance.

Under these assumptions one can show along the lines of \cite{DMMampa} that 
$$
\lim_{k\to +\infty}\frac{\mu_q\big(\omega, Q^k\big)}{|Q^k|}=C_q,
$$
in probability to a constant $C_q>0$, and conclude convergence of the functionals in probability to an integral functional with constant integrand $f_0(q)=C_q.$

As a first step,  defining
$$
\widetilde \mu_q(Q):={\mathbb E}(\mu_q(\omega,Q)),
$$
one can show along the lines of Section 5 that 
$$
\lim_{k\to +\infty}\frac{\widetilde \mu_q\big(\omega, Q^k\big)}{|Q^k|}=C_q,
$$
for some constant $C_q>0$.
Note that because of the invariance in law, Lemma   \ref{Proposition_periodic} holds for $\widetilde \mu$ but not for $\mu(\omega,\cdot)$ with $\omega$ fixed.

Now fix $ k_0\gg1$ so large  that 
$$\left|\frac{\widetilde \mu_q\big(Q^k\big)}{|Q^k|}-C_q\right|<\delta/4,\quad \textrm{for all}\; k\ge k_0
$$
 and  now fix $k>>k_0$, we use the construction in step 2 of the proof of 
Theorem \ref{LimiteImportante} to show that
$$
\frac{\mu(\omega, Q^k)}{|Q^k|}\le \frac{|Q^k|}{|Q^{k_0}|}\sum_{j\in S_k^{k_0} }\frac{ \mu\left(\omega, \tau_{jk_0}(Q^{k_0})\right)}{|Q^{k_0}|}+o\left( 1\right).
$$
The r.h.s. is a normalised sum over $(k/k_0)^{\mathcal Q}$ independent, identically distributed random variables with mean close to $C_q.$ By the weak law of large numbers, we 
have that for $\delta>0$ and $k$ sufficiently large the quantity
$$
\beta(\delta):={\mathbb P}\left(\left\{\omega\;\bigg|\; \frac{\mu(\omega, Q^k)}{|Q^k|}>C_q+\delta/4 \right\}\right)
$$ is small. Now define
$$
\alpha(\delta):={\mathbb P}\left(\left\{\omega\; \bigg|\;\frac{\mu(\omega, Q^k)}{|Q^k|}<C_q-\sqrt{\delta} \right\}\right).
$$
We have
\begin{eqnarray*}
C_q&\le& \frac{{\mathbb E}\left(\mu(\omega, Q^k)\right)}{|Q^k|}+\delta/2\\&\le& \alpha(\delta)(C_q-\sqrt{\delta}) +C_2(|q|^\alpha+1)\beta(\delta)\\
&+&
{\mathbb P}\left(\left\{\omega\,\bigg|\; C_q-\sqrt{\delta} \le \frac{\mu(\omega, Q^k)}{|Q^k|}<C_q+\delta/4\right\}\right)+\delta/2\\
&\le & \alpha(\delta)(C_q-\sqrt{\delta}) +C_2(|q|^\alpha+1)\beta(\delta) +(1-\alpha-\beta)(C_q+\delta/4)+\delta/2\\
&\le & C_q-\alpha(\delta)\sqrt{\delta} +(3/4) \delta+\beta(\delta)C_2(|q|^\alpha+1).
\end{eqnarray*}
As we can make $\beta(\delta) $ arbitrarily small  by choosing $k$ big, this implies that for such $k$ also $\alpha(\delta)\to 0$ in order to avoid the contradiction $C_q<C_q.$\\

\noindent{\bf Acknowledgments.} The first author was partially supported by EPSRC via grant EP/M028607/1.
The third and the fourth author are members of GNAMPA-INdAM and have been partially supported also by the Fondazione CaRiPaRo Project ``Nonlinear Partial Differential Equations: Asymptotic Problems and Mean-Field Games''. They warmly thank University of Cardiff for the kind hospitality.

\end{document}